\documentclass[12pt,reqno,twoside]{article}
\usepackage{amssymb,amsfonts,amsthm,amsmath}
\usepackage{cite}
\usepackage[left=1 in,top=1 in,right=1 in,bottom=1 in]{geometry}
\usepackage{enumitem}

\usepackage{color}

\usepackage{hyperref}

\newcommand{\beq}{\begin{equation}}
\newcommand{\eeq}{\end{equation}}
\newcommand{\beqs}{\begin{equation*}}
\newcommand{\eeqs}{\end{equation*}}
\newcommand{\ba}{\begin{array}}
\newcommand{\ea}{\end{array}}
\newcommand{\beas}{\begin{eqnarray*}}
\newcommand{\eeas}{\end{eqnarray*}}
\newcommand{\bea}{\begin{eqnarray}}
\newcommand{\eea}{\end{eqnarray}}
\newcommand{\bal}{\begin{align}}
\newcommand{\eal}{\end{align}}

\newcommand{\bals}{\begin{align*}}
\newcommand{\eals}{\end{align*}}

\newcommand{\al}{\alpha}

\newcommand{\R}{\ensuremath{\mathbb R}}
\newcommand{\C}{\ensuremath{\mathbb C}}
\newcommand{\N}{\ensuremath{\mathbb N}}
\newcommand{\Z}{\ensuremath{\mathbb Z}}

\newcommand{\inprod}[1]{\langle{#1}\rangle}
\newcommand{\doubleinprod}[1]{\langle\!\langle{#1}\rangle\!\rangle}

\newcommand{\bds}{\begin{displaystyle}}
\newcommand{\eds}{\end{displaystyle}}

\def\eqdef{\stackrel{\rm def}{=}}

\newcommand{\bvec}[1]{\mathbf{#1}}
\def\vecx{\bvec x}

\def\vecb{\bvec b}

\def\vece{\bvec e}
\def\vecu{\bvec u}
\def\vecv{\bvec v}
\def\vecx{\bvec x}
\def\vecw{\bvec w}

\def\veck{\bvec k}

\def\vecl{\bvec l}

\def\vecm{\bvec m}

\def\varep{\varepsilon}
\def\ddt{\frac{d}{dt}}

\newtheorem{theorem}{Theorem}[section]

\newtheorem{lemma}[theorem]{Lemma}
\newtheorem{corollary}[theorem]{Corollary}
\newtheorem{proposition}[theorem]{Proposition}
\newtheorem{definition}[theorem]{Definition}
\newtheorem*{assumption}{Basic Assumption}
\newtheorem*{notation}{Notation}

\theoremstyle{remark}
\newtheorem{remark}[theorem]{\bf{Remark}}

\def\mD{\mathcal D}

\numberwithin{equation}{section}

\title{\textbf{Asymptotic expansion for solutions of the Navier-Stokes equations with non-potential body forces}}
\author{Luan T. Hoang$^{1}$ and Vincent R. Martinez$^{2}$}

\date{\today}
\begin{document}
\maketitle

\begin{center}
$^{1}$Department of Mathematics and Statistics, Texas Tech University\\
Box 41042, Lubbock, TX 79409-1042, U.S.A.\\
Email address: \texttt{luan.hoang@ttu.edu}

\medskip
$^{2}$Mathematics Department, Tulane University\\
6823 St. Charles Ave, New Orleans, LA 70118, U.S.A.\\
Email address: \texttt{vmartin6@tulane.edu}
\end{center}

\begin{abstract}
We study the  long-time behavior of spatially periodic solutions of the Navier-Stokes equations in the three-dimensional space.
The body force is assumed to possess an asymptotic expansion or, resp., finite asymptotic approximation, in Sobolev-Gevrey spaces, as time tends to infinity,  in terms of polynomial and decaying exponential functions of time.
We establish an asymptotic expansion, or resp., finite asymptotic approximation, of the same type for the  Leray-Hopf weak solutions.
This extends previous results that were obtained in the case of potential forces, to the non-potential force case, where the body force may have different levels of regularity and asymptotic approximation.  
This expansion or approximation, in fact, reveals  precisely how the structure of the force influences the asymptotic behavior of the solutions.
\end{abstract}
\section{Introduction}\label{intro}
We study the  Navier-Stokes equations (NSE) for a viscous, incompressible fluid in the three-dimensional space, $\R^3$.
Let $\vecx\in \R^3$ and $t\in\R$  denote the space and time variables, respectively.
Let the (kinematic) viscosity be denoted by $\nu>0$, the  velocity vector field by $\vecu(\vecx,t)\in\R^3$, the pressure by $p(\vecx,t)\in\R$, and the body force by  $\mathbf f(\vecx,t)\in\R^3$. The NSE which describe the fluid's dynamics are given by
\begin{align}\label{nse}
\begin{split}
&\bds \frac{\partial \vecu}{\partial t}\eds  + (\vecu\cdot\nabla)\vecu -\nu\Delta \vecu = -\nabla p+\mathbf f \quad\text{on }\R^3\times(0,\infty),\\
&\textrm{div } \vecu = 0  \quad\text{on }\R^3\times(0,\infty).
\end{split}
\end{align}

The initial condition is 
\beq\label{ini}
\vecu(\vecx,0) = \vecu^0(\vecx),
\eeq 
where  $\vecu^0(\vecx)$ is a given divergence-free vector field.

In this paper, we focus on solutions $\vecu(\vecx, t)$ and $p(\vecx, t)$ which are $L$-periodic for some $L>0$.
Here, a function $g(\vecx)$ is $L$-periodic if
\beqs
g(\vecx+L\vece_j)=g(\vecx)\quad \textrm{for all}\quad \vecx\in \R^3,\ j=1,2,3,\eeqs
where  $\{\vece_1,\vece_2,\vece_3\}$ is the standard basis of $\R^3$.

By the remarkable Galilean transformation, we can assume also that $\vecu(\vecx,t)$, for all $t\ge 0$,  has zero average over the domain $\Omega=(-L/2, L/2)^3$.
A function $g(\vecx)$ is said to have zero average over $\Omega$ if 
\beq \label{Zacond} 
\int_\Omega g(\vecx)d\vecx=0.
\eeq

By rescaling the spatial and time variables, we assume throughout, without loss of generality, that  $L=2\pi$ and $\nu =1$.

\begin{notation} We will use the following standard notation.
\begin{enumerate}[label={\rm (\alph*)}]
 \item In studying dynamical systems in infinite dimensional spaces, we denote, regarding \eqref{nse} and \eqref{ini},  $u(t)=\vecu(\cdot,t)$, $f(t)=\mathbf f(\cdot,t)$, and $u^0=\vecu^0(\cdot)$. 
 \item For non-negative functions $h(t)$ and $g(t)$, we write 
	\beqs
		h(t)=\mathcal{O}(g(t))\ \text{as $t\to\infty$}\quad \text{if there exist $T,C>0$ such that }\ h(t)\leq Cg(t),\ \forall t>T.
	\eeqs
 \item The Sobolev spaces on $\Omega$ are denoted by $H^m(\Omega)$ for $m=0,1,2,\ldots$, each consists of functions on $\Omega$ with distributional derivatives up to order $m$ belonging to $L^2(\Omega)$.
\end{enumerate}
\end{notation}

The type of asymptotic expansion that we study here is defined, in a general setting, as follows.

\begin{definition}\label{expanddef}
Let  $X$ be a real vector space.

{\rm (a)} An $X$-valued polynomial is a function $t\in \R\mapsto \sum_{n=1}^d a_n t^n$,
for some $d\ge 0$, and $a_n$'s belonging to $X$.

{\rm (b)} When $(X,\|\cdot\|)$ is a normed space, a function $g(t)$ from $(0,\infty)$ to $X$ is said to have the asymptotic expansion
	\beq \label{gexpand}
g(t) \sim \sum_{n=1}^\infty g_n(t)e^{-nt} \text{ in } X,
\eeq
where $g_n(t)$'s are $X$-valued polynomials, if for all $N\geq1$, there exists $\varep_N>0$ such that
\beq \label{grem}
\Big\|g(t)- \sum_{n=1}^N g_n(t)e^{-nt}\Big\|=\mathcal O(e^{-(N+\varep_N)t})\ \text{as }t\to\infty.
\eeq
\end{definition}

This article aims at studying the asymptotic behavior of the solution $\vecu(\vecx,t)$ as $t\to\infty$ for a certain class of forces $\mathbf f(\vecx,t)$.
The case when $\mathbf f$ is a potential force, i.e., $\mathbf f(\vecx,t)=-\nabla \phi(\vecx,t)$, for some scalar function $\phi$, has been well-studied.  For instance, it is well-known that any Leray-Hopf weak solution becomes regular eventually and decays in the $H^1(\Omega)$-norm exponentially.
For more precise asymptotic behavior, Dyer and Edmunds \cite{DE1968} proved that a non-trivial, regular solution is also bounded below by an exponential function.
Foias and Saut \cite{FS84a} then proved that in bounded or periodic domains, a non-trivial,  regular solution decays exponentially at an exact rate which is an eigenvalue of the Stokes operator. 
Furthermore,  they established in  \cite{FS87} that  for such a  solution $u(t)$, the following asymptotic expansion holds, in the sense of Definition \ref{expanddef}, in Sobolev spaces $H^m(\Omega)^3$, for all $m\ge 0$:
\beq \label{expand}
u(t) \sim \sum_{n=1}^\infty q_n(t)e^{-nt},
\eeq
where $q_n(t)$'s are  unique polynomials in $t$ 
with trigonometric polynomial values.  


Recently, in \cite{HM1}, it was shown by the authors that the expansion in fact holds in Gevrey spaces.
More precisely, for any $\sigma>0$  and $m\in\N$, there exists an $\varep_N>0$, for each integer $N>0$, such that
\beq\label{vHM}
\Big\|e^{\sigma (-\Delta)^{1/2}} \Big(u(t)- \sum_{n=1}^N q_n(t)e^{-nt}\Big)\Big\|_{H^m(\Omega)^3}=\mathcal O\big(e^{-(N+\varepsilon_N)t}\big)\textrm{ as }t\to\infty.
\eeq
Note that the (Gevrey) norm in estimate \eqref{vHM} is much stronger than the standard Sobolev norm in $H^m(\Omega)$.
More importantly, the simplified approach in \cite{HM1} allows the proof to be applied to wider classes of equations; that approach will be adopted in this paper.

Regarding the case of potential forces, the interested reader is referred to \cite{FS83,FS84a,FS84b,FS87,FS91} for deeper studies on the asymptotic expansion, its associated normalization map, and invariant nonlinear manifolds; for the associated (Poincar\'e-Dulac) normal form, see  \cite{FHOZ1,FHOZ2,FHS1}; for its applications to statistical solutions of the NSE, decaying turbulence, and analysis of helicity, see \cite{FHN1,FHN2}; for a result in the whole space $\mathbb R^3$, see  \cite{KukaDecay2011}.

\medskip
The main goal of this paper is to establish \eqref{expand} when $\mathbf f$ is \textit{not} a potential function.
To understand the result without calling for technical details, we state it here as a `meta theorem'; the rigorous version will be provided by Theorem \ref{mainthm}, whose proof will then be presented in section \ref{pfsec}.

\begin{theorem}[Meta Theorem]\label{meta}
Assume that the body force has an asymptotic expansion 
\beq\label{fexpand}
f(t)\sim \sum_{n=1}^\infty f_n(t)e^{-nt},
\eeq
in some appropriate functional space. 
Then any Leray-Hopf weak solution $u(t)$ of \eqref{nse} and \eqref{ini} admits an asymptotic expansion of the form \eqref{expand} in the same space.
\end{theorem}

On the other hand, when $f(t)$ satisfies a finitary version of \eqref{fexpand}, then we obtain a corresponding finite asymptotic approximation result in Theorem \ref{finitetheo}.
More specifically, if the right-hand side of \eqref{fexpand} is a finite sum, then we show that the corresponding solution, $u(t)$, admits a finite sum approximation of the same type analogous to \eqref{expand}.
Theorem \ref{finitetheo} is also suitable for the case when the force is not smooth, but belongs only to a pre-specified Sobolev class.

Theorem \ref{meta} develops one of the possible avenues of extending the Foias-Saut theory \cite{FS87}  for the NSE to the  non-potential force case. The finite sum version, Theorem \ref{finitetheo}, which is a new feature only for the non-potential force case, allows asymptotic analysis of the solution even when the force has restricted regularity and only limited information on the force's asymptotic behavior is known.  Our analysis, moreover, explicitly indicates how each term in the expansion of the force integrates into the expansion of the solution. 
We emphasize that the assumptions that we do impose on the force are very natural. This point is not so obvious because if one directly adopts the argument of Foias-Saut in  \cite{FS87}, one would then be led to impose conditions on \textit{time-derivatives of all orders} on the force.  We, nevertheless, are able to avoid these conditions and ultimately establish the claimed expansion by applying the refined approach, in the case of periodic domains, initiated in \cite{HM1}. 
We also remark that we obtain a welcome technical improvement over \cite{HM1} in obtaining the expansion in Sobolev spaces, when Gevrey regularity is not available; it requires a more elaborate bootstrapping process which is carried out in Part II of the proof of Proposition \ref{theo23}.

The results obtained in this paper  are the first steps of the larger program of understanding the relation between the asymptotic expansion and the external body force.  On the other hand, in spite of assuming the force to have regular modes of decay, further understanding of the solution's resulting expansion and its consequences will shed insights into the nonlinear structure of the NSE, and decaying turbulence theory.
It is worth mentioning at this point that the global existence of regular solutions and the uniqueness of the global weak solutions of the 3D NSE remain outstanding open problems. 
However, our result details the asymptotic behavior of any Leray-Hopf weak solution, regardless of the resolution to either of these issues. 
Moreover, to the best of the authors' knowledge, there have been no numerical studies made to determine the polynomials $q_n(t)$ in \eqref{expand} as of yet. By extending the expansion to accommodate non-potential forces,  our result should facilitate the formulation and testing of possible numerical algorithms for their computation. Indeed, one can attempt to compute the expansion of \emph{explicit} solutions, particularly, those for which the nonlinear term in NSE does not vanish. These solutions are quite easy to generate when certain, specific forces are used at the onset, but harder to find 
when the projected force in the functional equation \eqref{fctnse} is \emph{given and fixed}, albeit zero, as in the case of potential forces (see \eqref{potential}).

For the structure of this paper, Section \ref{bkgmain} lays the necessary background for formulating the result. 
The main theorems are  Theorems \ref{mainthm} and \ref{finitetheo}, while Corollary \ref{Vcor} emphasizes the scenario of finite-dimensional polynomial coefficients in \eqref{fexpand}, and the construction of the corresponding polynomials appearing \eqref{expand} as solutions to finite-dimensional, linear ordinary differential equations.
Section \ref{Gdecay} contains explicit estimates that ultimately furnish ``eventual" exponential decay for the weak solutions in Sobolev and Gevrey spaces, see Propositions \ref{theo22} and \ref{theo23}. They are used crucially in section \ref{pfsec}, which is devoted to the proofs of our main results.

\section{Background and main results}\label{bkgmain}

The space $L^2(\Omega)^3$ of square (Lebesgue) integrable vector fields on $\Omega$ is a Hilbert space with the standard inner product $\inprod{\cdot,\cdot}$ and norm $|\cdot
|$ defined by 
$$
\inprod{u,v}=\int_\Omega \vecu(\vecx)\cdot \vecv(\vecx) d\vecx
\quad\text{and}\quad  |u|
        =\inprod{u,u}^{1/2}\quad\text{for } u=\vecu(\cdot),\ v=\vecv(\cdot).$$

We note that $|\cdot|$ is also used to  denote  the absolute value, modulus, and, more generally, the Euclidean norm in $\R^n$ and $\C^n$, for $n\in\N$. Nonetheless, its meaning will be made clear by the context.

Let $\mathcal{V}$ be the set of all $L$-periodic trigonometric polynomial vector fields which are divergence-free and  have zero average over $\Omega$.  
Define
$$H, \text{ resp. } V\ =\text{ closure of }\mathcal{V} \text{ in }
L^2(\Omega)^3, \text{ resp. } H^1(\Omega)^3.$$

We use the following embeddings and identification
$$V\subset H=H'\subset V',$$ 
where each space is dense in the next one, and the embeddings are compact.


Let $\mathcal{P}$ denote the orthogonal (Leray) projection in $L^2(\Omega)^3$ onto $H$. Explicitly,
	\begin{align*}
		\mathcal{P}\Big(\sum_{\veck\ne \mathbf 0} 
\widehat \vecu(\veck)e^{i\veck\cdot \vecx}\Big) = \sum_{\veck\ne \mathbf 0} \Big\{\widehat\vecu(\veck)-\Big(\widehat\vecu(\veck)\cdot \frac{\veck}{|\veck|}\Big)\frac{\veck}{|\veck|}\Big\}e^{i\veck\cdot \vecx} .
	\end{align*}

We define the Stokes operator $A:V\to V'$ by	
\beqs
\inprod{A\vecu,\vecv}_{V',V}=
\doubleinprod{\vecu,\vecv}
\eqdef \sum_{i=1}^3 \inprod{ \frac{\partial \vecu}{\partial x_i} , \frac{\partial \vecv}{\partial x_i} },\quad \text{for all } \vecu,\vecv\in V.
\eeqs

As an unbounded operator on $H$, the operator $A$ has the domain $\mD(A)=V\cap H^2(\Omega)^3$, and 
\beqs A\vecu = - \mathcal{P}\Delta \vecu=-\Delta \vecu\in H, \quad \textrm{for all}\quad \vecu\in\mD(A).
\eeqs
The last identity is due to the periodic boundary conditions.

It is known that the spectrum  of the Stokes operator $A$ is $\sigma(A)=\{\lambda_j:j\in \N\}$,
where $\lambda_j$ is strictly increasing in $j$, and is an  eigenvalue of $A$.
In fact, for each $j\in\N$, $\lambda_j=|\veck|^2$ for some $\veck\in\Z^3\setminus \{\mathbf 0\}$.  Note that since $\sigma(A)\subset \N$ and $\lambda_1=1$, the additive semigroup generated by $\sigma(A)$ is equal to $\N$.

\medskip 
For $\alpha,\sigma \in \R$ and  $\vecu=\sum_{\veck\ne \mathbf 0} 
\widehat \vecu(\veck)e^{i\veck\cdot \vecx}$, define
$$A^\alpha \vecu=\sum_{\veck\ne \mathbf 0} |\veck|^{2\alpha} \widehat \vecu(\veck)e^{i\veck\cdot 
\vecx},$$
$$A^\alpha e^{\sigma A^{1/2}} \vecu=\sum_{\veck\ne \mathbf 0} |\veck|^{2\alpha}e^{\sigma 
|\veck|} \widehat \vecu(\veck)e^{i\veck\cdot 
\vecx}.$$

We then define the  Gevrey spaces by
\beqs
G_{\alpha,\sigma}=\mD(A^\alpha e^{\sigma A^{1/2}} )\eqdef \{ \vecu\in H: |\vecu|_{\alpha,\sigma}\eqdef |A^\alpha 
e^{\sigma A^{1/2}}\vecu|<\infty\},
\eeqs
and the domain of the fractional operator $A^\al$ by 
\beqs 
\mD(A^\alpha)=G_{\alpha,0}=\{ \vecu\in H: |A^\alpha \vecu|=|\vecu|_{\alpha,0}<\infty\}.
\eeqs

Thanks to the zero-average condition \eqref{Zacond}, the norm $|A^{m/2}\vecu|$ is equivalent to $\|\vecu\|_{H^m(\Omega)^3}$ on the space $\mathcal D(A^{m/2})$ for $m=0,1,2,\ldots$

Note that $\mD(A^0)=H$,  $\mD(A^{1/2})=V$, and  $\|\vecu\|\eqdef |\nabla \vecu|$ is equal to $|A^{1/2}\vecu|$ for $\vecu\in V$.
Also, the spaces $G_{\alpha,\sigma}$ are decreasing in $\alpha$ and $\sigma$.

Denote for $\sigma\in\R$ the space 
\beqs
E^{\infty,\sigma}=\bigcap_{\alpha\ge 0} G_{\alpha,\sigma}=\bigcap_{m\in\N } G_{m,\sigma}.
\eeqs

We will say that an asymptotic expansion \eqref{gexpand} holds in  $E^{\infty,\sigma}$ if it holds in $G_{\alpha,\sigma}$ for all $\alpha\ge 0$.

Let us also denote by $\mathcal{P}^{\alpha,\sigma}$ the space of $G_{\alpha,\sigma}$-valued polynomials in case $\alpha\in \R$, and 
the space of $E^{\infty,\sigma}$-valued polynomials in case $\alpha=\infty$.

\medskip

We  define the bilinear mapping $B:V\times V\to V'$, which is associated with
the nonlinear term in the NSE,  by
\beqs
\inprod{B(\vecu,\vecv),\vecw}_{V',V}=b(\vecu,\vecv,\vecw)\eqdef \int_\Omega ((\vecu\cdot \nabla) \vecv)\cdot \vecw\, d\vecx, \quad \textrm{for all}\quad \vecu,\vecv,\vecw\in V.
\eeqs 
In particular,
\beqs
B(\vecu,\vecv)=\mathcal{P}((\vecu\cdot \nabla) \vecv), \quad \textrm{for all}\quad \vecu,\vecv\in\mD(A).
\eeqs
More precisely, for $\vecu=\sum_{\veck\ne 0} 
\widehat \vecu(\veck)e^{i\veck\cdot \vecx}$ and $\vecv=\sum_{\veck\ne 0} 
\widehat \vecv(\veck)e^{i\veck\cdot \vecx}$,
\beqs
B(\vecu,\vecv)
= \sum_{\veck\ne \mathbf 0}  \Big\{\widehat\vecb(\veck)-\Big(\widehat\vecb(\veck)\cdot \frac{\veck}{|\veck|}\Big)\frac{\veck}{|\veck|} \Big\}e^{i\veck\cdot 
\vecx},
\quad\text{where }\widehat\vecb(\veck)=\sum_{\vecm+\vecl=\veck} 
i  (\widehat\vecu(\vecm)\cdot \vecl)\widehat\vecv(\vecl).
\eeqs

It is clear that
\beq\label{BVV}
B(\mathcal V,\mathcal V)\subset \mathcal V.
\eeq

By applying the Leray projection $\mathcal{P}$ to  \eqref{nse} and \eqref{ini}, 
we rewrite the initial value problem for NSE in the functional form as
\beq\label{fctnse}
\frac{du(t)}{dt} + Au(t) +B(u(t),u(t))=\mathcal P f(t) \quad \text{ in } V' \text{ on } (0,\infty),
\eeq
with the initial data 
\beq\label{uzero} 
u(0)=u^0\in H.
\eeq
(See e.g. \cite{LadyFlowbook69,CFbook,TemamAMSbook} for more details.)

Because of the projection $\mathcal P$ on the right-hand side of  \eqref{fctnse}, it is convenient to assume, without loss of generality that the force belongs to $H$. Then, we have
$$\mathcal P f(t)=f(t)\text{ in  \eqref{fctnse}}.$$

In the case $\mathbf f(\vecx,t)$ is a potential force, then, by the Helmholtz-Leray decomposition, 
\beq \label{potential}
\mathcal P f(t)\equiv 0 \text{ in the functional equation \eqref{fctnse}}.
\eeq

In dealing with weak solutions of \eqref{fctnse}, we follow the presentation in \cite{FMRTbook} and use the results there.

\begin{definition}\label{lhdef}
Let $f\in L^2_{\rm loc}([0,\infty),H)$.
A \emph{Leray-Hopf weak solution} $u(t)$ of \eqref{fctnse} is a mapping from $[0,\infty)$ to $H$ such that 
\beq\label{lh:wksol}
u\in C([0,\infty),H_{\rm w})\cap L^2_{\rm loc}([0,\infty),V),\quad u'\in L^{4/3}_{\rm loc}([0,\infty),V'),
\eeq
and satisfies 
\beq\label{varform}
\ddt \inprod{u(t),v}+\doubleinprod{u(t),v}+b(u(t),u(t),v)=\inprod{f(t),v}
\eeq
in the distribution sense in $(0,\infty)$, for all $v\in V$, and the energy inequality
\beq\label{Lenergy}
\frac12|u(t)|^2+\int_{t_0}^t \|u(\tau)\|^2d\tau\le \frac12|u(t_0)|^2+\int_{t_0}^t \langle f(\tau),u(\tau)\rangle d\tau
\eeq
holds for $t_0=0$ and almost all $t_0\in(0,\infty)$, and all $t\ge t_0$.  

We will say that a Leray-Hopf weak solution $u(t)$ is \emph{regular} if $u\in C([0,\infty),V)$.
\end{definition}

Above, $H_{\rm w}$ is the topological vector space $H$ with the weak topology.

This definition of the Leray-Hopf weak solutions with the choice of the energy inequality \eqref{Lenergy} is, in fact, equivalent to the weak solutions used in \cite[Chapter II, section 7]{FMRTbook}, see e.g. Remark 1(e) of  \cite{FRT2010} for the explanations.

\begin{assumption}
It is assumed throughout the paper that the function $f(t)$ belongs to $L^\infty_{\rm loc}([0,\infty),H)$.
\end{assumption}

This assumption serves to guarantee the existence of the Leray-Hopf weak solutions for any $u^0\in H$, see e.g. \cite{FMRTbook}. We note that, later in the paper,  we will further impose that the force, $f(t)$, decays in time, for instance, in Sobolev norms.

Regarding the problem of finding asymptotic expansions for solutions of the NSE, it is natural, at this stage,  to study the class of functions $f(t)$ which have similar asymptotic behavior as $u(t)$ in \eqref{expand}.
We specify the condition on $f(t)$ more precisely in the next theorem, which is our first main result.

\begin{theorem}[Asymptotic expansion] \label{mainthm}
Assume that there exist a number $\sigma_0\geq0$ and polynomials $f_n\in\mathcal{P}^{\infty,\sigma_0}$, for all $n\ge 1$, such that $f(t)$ has the asymptotic expansion
\beq\label{forcexpand}
f(t)\sim \sum_{n=1}^\infty f_n(t)e^{-nt}\quad \text{in }E^{\infty,\sigma_0} .
\eeq

Let $u(t)$ be a Leray-Hopf weak solution of \eqref{fctnse} and \eqref{uzero}.
Then there exist polynomials $q_n\in\mathcal{P}^{\infty,\sigma_0}$, for all $n\ge 1$, such that $u(t)$ has the asymptotic expansion
 \beq\label{uexpand}
u(t)\sim  \sum_{n=1}^\infty  q_n(t) e^{-nt}\quad \text{in }E^{\infty,\sigma_0} .
 \eeq
 
Moreover, the mappings
\beq\label{uF}
u_n(t)\eqdef q_n(t) e^{-nt} \quad\text{and}\quad F_n(t)\eqdef f_n(t)e^{-nt},
\eeq
satisfy the following ordinary differential equations in the space $E^{\infty,\sigma_0}$
\beq\label{unODE}
\ddt u_n(t) + Au_n(t)  +\sum_{\stackrel{k,m\ge 1}{k+m=n}}B(u_k(t),u_m(t))= F_n(t),\quad t\in\R,
\eeq
for all $n\ge 1$.
\end{theorem}

Regarding equation \eqref{unODE}, when $n=1$, the sum on its left-hand side is empty, hence the equation reads as
\beq\label{u1ODE}
\ddt u_1(t) + Au_1(t) = F_1(t).
\eeq

\begin{remark}\label{betterem}
Observe that since the expansion \eqref{forcexpand} is an infinite sum, it immediately implies the following remainder estimate:
\begin{align*}
\Big|f(t)-\sum_{n=1}^N f_n(t)e^{-nt}\Big|_{\alpha,\sigma_0}
&\le \Big|f_{N+1}(t)e^{-(N+1)t}\Big|_{\alpha,\sigma_0}+\Big|f(t)-\sum_{n=1}^{N+1} f_n(t)e^{-nt}\Big|_{\alpha,\sigma_0}\\
&=\mathcal O(e^{-(N+\varep)t})+\mathcal O(e^{-(N+1+\delta_{N+1},\alpha)t}),
\end{align*}
which holds for all $N\geq1$, $\alpha\ge 0$, $\varep\in(0,1)$ and some $\delta_{N+1,\alpha}\in(0,1)$. Therefore, we have for all $N\geq1$, $\alpha\ge0$  that
\beq\label{fep}
\Big|f(t)-\sum_{n=1}^N f_n(t)e^{-nt}\Big|_{\alpha,\sigma_0}
=\mathcal O(e^{-(N+\varep)t}) \quad\text{as }t\to\infty,\quad \forall\varep\in(0,1).
\eeq
Similarly, the expansion \eqref{uexpand} implies for any $N\ge 1$ and $\alpha\ge0$ that 
 \beq\label{uqa}
\Big |u(t)-\sum_{n=1}^N q_n(t) e^{-nt}\Big |_{\alpha,\sigma_0} =\mathcal O(e^{-(N+\varep)t} )\quad\text{as } t\to\infty,\quad \forall\varep\in(0,1).
 \eeq
(In fact, the same argument applies to the general expansion \eqref{gexpand} so that $\varep_N$ in \eqref{grem} can be taken arbitrarily in $(0,1)$.)
\end{remark}

\begin{remark}\label{mainrmk} The following additional remarks on Theorem \ref{mainthm} are in order.
\begin{enumerate}[label={\rm (\alph*)}]
 \item We do not require the time derivative $d^m f/dt^m$, \emph{for all} $m\in\N$, to have any kind of expansion. 
 Indeed, this rather stringent requirement would have been imposed if one adapts Foias-Saut's original proof. 
 Our relaxation of this condition is owed to the higher regularity of the solutions for large time, in the particular case of periodic domains, either in Sobolev or Gevrey spaces.
In the case of Sobolev spaces ($\sigma_0=0$), it requires a  bootstrapping scheme to gradually increase the regularity to any needed level.
In the case of the Gevrey spaces ($\sigma_0>0$), the effect is immediate, hence the proof is shorter.
(For the related Gevrey norm techniques, see \cite{FT-Gevrey, HM1, OliverTiti2000} and references therein.)

 \item The equations \eqref{unODE} determine the polynomials $q_n(t)$'s and indicate the interactions on all scales between the  body force and the nonlinear terms in NSE. 
 Even though the solution $u(t)$ decays to zero, such interactions are complicated.

 \item Condition \eqref{forcexpand} is easily satisfied for any finite sum
\beqs
f(t)=\sum_{n=1}^N f_n(t)e^{-nt},
\eeqs
for some fixed $N\ge 1$, and polynomials $f_n(t)$'s belonging to $\mathcal V$. Even in this case, the result in Theorem \ref{mainthm} is new and the expansion \eqref{uexpand} for $u(t)$ can still be an infinite sum.
 
 \item If the expansion \eqref{forcexpand} for $f(t)$ holds in $G_{0,\sigma_1}$ for some $\sigma_1>0$, then it holds in  $E^{\infty,\sigma_0}$ for any $\sigma_0\in (0,\sigma_1)$, and hence, by Theorem \ref{mainthm},  the solution $u(t)$ admits the expansion \eqref{uexpand} in  $E^{\infty,\sigma_0}$ for any $\sigma_0\in (0,\sigma_1)$.

\item The equations \eqref{unODE} in fact are  {\it linear} systems of ordinary differential equations (ODEs) in infinite-dimensional spaces. They form an integrable system in the sense that it can be recursively solved by the variation of constants formula. Moreover, each solution $u_n(t)$ of the form \eqref{uF} is uniquely determined provided that $R_n u_n(0)=\xi_n\in R_nH$ is given. 
\end{enumerate}
\end{remark}

In the following simple scenario, the ODEs  \eqref{unODE} are finite-dimensional systems, which make them more accessible for deeper study, and, in particular, for numerical computations.

\begin{corollary}\label{Vcor}
If all $f_n(t)$'s in Theorem \ref{theo22} are $\mathcal V$-valued polynomials, 
then so are the polynomials $q_n(t)$'s in the expansion \eqref{uexpand}, and  consequently, the equations \eqref{unODE} are systems of linear ODEs in finite-dimensional spaces.
\end{corollary}

Next is the paper's second main result which deals with the case when $f(t)$ does not possess an asymptotic expansion \eqref{forcexpand}, but rather a \emph{finite asymptotic approximation}. Then it is proved that  the weak solution also admits a finite asymptotic approximation of the same type. 

\begin{theorem}[Finite asymptotic approximation]\label{finitetheo}
Suppose there exist an integer $N_*\ge1$, real numbers $\sigma_0\ge 0$, $\mu_*\ge \alpha_*\ge N_*/2$,  and, for any $1\le n\le N_*$,  numbers $\delta_n\in(0,1)$ and polynomials $f_n\in\mathcal{P}^{\mu_n,\sigma_0}$,  such that  
\beq \label{ffinite}
\Big|f(t)-\sum_{n=1}^N f_n(t)e^{-nt}\Big|_{\alpha_N,\sigma_0}=\mathcal O(e^{-(N+\delta_N)t}) \quad\text{as }t\to\infty,
\eeq
for $1\le N\le N_*$, where
\beqs
\mu_n=\mu_*-(n-1)/2,\quad \alpha_n=\alpha_*-(n-1)/2.
\eeqs

Let $u(t)$ be a Leray-Hopf weak solution of \eqref{fctnse} and \eqref{uzero}.

\begin{enumerate}[label={\rm{(\roman*)} }]
 \item  Then there exist polynomials $q_n\in\mathcal{P}^{\mu_n+1,\sigma_0}$, for  $1\le n\le N_*$, 
such that one has  for $1\le N\le N_*$ that
 \beq \label{ufinite}
\Big |u(t)-\sum_{n=1}^N q_n(t)e^{-nt}\Big|_{\alpha_N,\sigma_0}=\mathcal O(e^{-(N+\varep)t})\quad\text{as }t\to\infty,\quad\forall\varep\in(0,\delta_N^*),
 \eeq
 where
 $ \delta_N^*=\min\{\delta_1,\delta_2,\ldots,\delta_N\}$.
 
Moreover, the ODEs  \eqref{unODE} hold in the corresponding space $G_{\mu_n,\sigma_0}$ for  $1\le n\le N_*$.

\item In particular, if all $f_n(t)$'s belong to $\mathcal V$, resp., $E^{\infty,\sigma_0}$, then so do all $q_n(t)$'s, and the ODEs \eqref{unODE} hold in $\mathcal V$, resp., $E^{\infty,\sigma_0}$.
\end{enumerate}
\end{theorem}

Regarding the finite approximation \eqref{ffinite}, in addition to its sum having only finitely many terms ($N\le N_*$), the force $f(t)$ now has much less regularity compared to the expansion \eqref{forcexpand}. This, therefore, determines the regularity of the solution $u(t)$ and its asymptotic approximation \eqref{ufinite}. Such dependence was worked out in detail in the above theorem.

It is worth noticing that Theorem \ref{finitetheo} is stronger than Theorem \ref{mainthm}. Nevertheless, the proof of Theorem \ref{finitetheo}, in  the current presentation,  is adapted from that of Theorem \ref{mainthm}.

\section{Exponential decay in Gevrey and Sobolev spaces}\label{Gdecay}

This section prepares for  the proofs in section \ref{pfsec}. Particularly, we will derive the exponential decay for weak solutions in both Gevrey  and Sobolev spaces.

First, we have a few basic inequalities concerning the Sobolev and Gevrey norms.
It is elementary to see that 
\beqs 
\max_{x\ge 0} (x^{2\alpha}e^{-\sigma x})=\Big(\frac{2\alpha}{e\sigma }\Big)^{2\alpha} \text{ for any } \sigma,\alpha>0.
\eeqs 

Applying this inequality, one easily obtains
\beq\label{als}
|A^\alpha u|=|(A^\alpha e^{-\sigma A^{1/2}})e^{\sigma A^{1/2}}u| \le \Big(\frac{2\alpha}{e\sigma}\Big)^{2\alpha} |e^{\sigma A^{1/2}}u|\quad \forall \alpha,\sigma>0.
\eeq

We recall a key estimate for the Gevrey norm of the bilinear form (cf. Lemma 2.1 in \cite{HM1}), which is based mainly on the work of Foias and Temam \cite{FT-Gevrey}.  

\begin{lemma}
\label{nonLem}
Let $\sigma\ge 0$ and $\alpha\ge 1/2$.  There exists an absolute constant $K>1$, independent of $\alpha,\sigma$, such that
\beq\label{AalphaB} 
|B(u,v)|_{\alpha ,\sigma }\le K^\alpha  |u|_{\alpha +1/2,\sigma } \, |v|_{\alpha +1/2,\sigma}\quad \forall u,v\in G_{\alpha+1/2,\sigma}.
\eeq
\end{lemma}

Although inequality \eqref{AalphaB} is not sharp, it is very convenient for our calculations below and will be sufficient for our purposes.

As a consequence of Lemma \ref{nonLem}, we have
\beq\label{BGG}
B(G_{\alpha+1/2,\sigma},G_{\alpha+1/2,\sigma})\subset G_{\alpha,\sigma}\quad\text{for } \alpha\ge 1/2,\ \sigma\ge 0,
\eeq
\beq\label{BEE}
B(E^{\infty,\sigma},E^{\infty,\sigma})\subset E^{\infty,\sigma}\quad\text{for }\sigma\ge 0.
\eeq

Next, we prove a small data result, which establishes the global existence of the solution in Gevrey spaces and its exponential decay as time goes to infinity. 

\begin{proposition} \label{theo22}
Let $\delta\in (0,1), \lambda\in(1-\delta,1]$ and $\sigma\ge 0, \alpha\ge1/2$.
Define the positive numbers $C_0=C_0(\alpha,\delta)$ and $C_1=C_1(\alpha,\delta,\lambda)$ by 
\beqs
\begin{cases}
C_0=\frac\delta{6K^\alpha},\
 C_1=\frac2{\sqrt3}\,\sqrt{\delta(\lambda-1+\delta)}\, C_0&\text{if }\sigma>0,\\
C_0=\frac\delta{4K^\alpha},\
 C_1=\sqrt{2}\,\sqrt{\delta(\lambda-1+\delta)}\, C_0&\text{if }\sigma=0.
\end{cases}
\eeqs 
Suppose
\beq\label{usmall}
|A^\alpha u^0|\le C_0,
\eeq
and
\beq\label{fta}
|f(t)|_{\alpha-1/2,\sigma}\le C_1e^{-\lambda t},\quad\forall t\ge0.
\eeq

Then there exists a unique solution $u(t)$ of \eqref{fctnse} and \eqref{uzero} that satisfies 
$$u\in C([0,\infty),\mathcal D(A^\alpha))$$ and 
\beq\label{uest}
|u(t)|_{\alpha,\sigma}\le \sqrt{2}C_0e^{-(1-\delta)t},\quad \forall t\ge t_*,
\eeq 
where $t_*=6\sigma/\delta$.
Moreover, one has for all $t\ge t_*$ that
 \beq\label{intAa}
 \int_t^{t+1} |u(\tau)|_{\alpha+1/2,\sigma}^2d\tau
 \le  \frac{3C_0^2}{2(1-\delta)}e^{-2(1-\delta)t}.
 \eeq
\end{proposition}
\begin{proof}
While the estimates below are formal, they can be justified by performing them at the level of the Galerkin approximation and then passing to the limit. The estimates will hold for the unique, regular solution $u(t)$.

\textbf{Part I: case $\sigma>0$.} Let $\varphi(t)$ be a function in $C^\infty(\R)$ such that 
$$\varphi((-\infty,0])=\{0\},\quad
\varphi([0,t_*])=[0,\sigma],\quad
\varphi([t_*,\infty))=\{\sigma\},$$
and 
$$0< \varphi'(t)< 2\sigma/t_*=\delta/3\quad \text{for all }t\in(0,t_*).$$

From equation \eqref{fctnse}, we have
\beq\label{daeu}
\ddt (A^{\alpha}e^{\varphi(t) A^{1/2}}u(t)) =A^{\alpha}e^{\varphi(t) A^{1/2}}(-Au-B(u,u)+f)
+ \varphi'(t)A^{1/2}A^{\alpha}e^{\varphi(t) A^{1/2}} u.
\eeq

Taking inner product of the equation \eqref{daeu} with $A^{\alpha}e^{\varphi(t) A^{1/2}}u(t)$ gives
\begin{align*}
&\frac12\ddt |u|_{\alpha,\varphi(t)}^2   + |A^{1/2}u|_{\alpha,\varphi(t)}^2
=\varphi'(t)\langle A^{2\alpha+1/2}e^{2\varphi(t) A^{1/2}}u,u\rangle\\ 
&\quad -\langle A^{\alpha}e^{\varphi(t) A^{1/2}}B(u,u),A^{\alpha}e^{\varphi(t) A^{1/2}}u\rangle+ \langle A^{\alpha-1/2}e^{\varphi(t) A^{1/2}}f,A^{\alpha+1/2}e^{\varphi(t) A^{1/2}}u\rangle.
\end{align*}
Applying the Cauchy-Schwarz inequality, then Lemma \ref{nonLem} to the second term on the right-hand side, we obtain
\begin{align*}
&\frac12\ddt |u|_{\alpha,\varphi(t)}^2   + |A^{1/2}u|_{\alpha,\varphi(t)}^2\\
& \le \varphi'(t) |u|_{\alpha+1/2,\varphi(t)}^2+K^\alpha |A^{1/2}u|_{\alpha,\varphi(t)}^2 |u|_{\alpha,\varphi(t)}
+ |f(t)|_{\alpha-1/2,\varphi(t)}|u|_{\alpha+1/2,\varphi(t)}, \\
&\le \frac\delta3 |u|_{\alpha+1/2,\varphi(t)}^2+K^\alpha |A^{1/2}u|_{\alpha,\varphi(t)}^2 |u|_{\alpha,\varphi(t)}
+ \frac3{4\delta}|f(t)|_{\alpha-1/2,\varphi(t)}^2 +\frac\delta3|u|_{\alpha+1/2,\varphi(t)}^2.
\end{align*}
This implies
\beq\label{s1}
\frac12\ddt |u|_{\alpha,\varphi(t)}^2 + \Big(1-\frac{2\delta}3 -K^\alpha |u|_{\alpha,\varphi(t)}\Big)|A^{1/2}u|_{\alpha,\varphi(t)}^2 \le \frac3{4\delta}|f(t)|_{\alpha-1/2,\sigma}^2.
\eeq

Let $T\in(0,\infty)$. Note that $|u(0)|_{\alpha,\varphi(0)}=|A^\alpha u^0|<2C_0$. Assume that
\beq\label{uT}
|u(t)|_{\alpha,\varphi(t)}\le 2C_0,\quad \forall t\in[0,T).
\eeq
Then for $t\in (0,T)$, we have from \eqref{s1} and \eqref{fta} that 
\beq\label{s2}
\ddt |u|_{\alpha,\varphi(t)}^2 + 2(1-\delta)|A^{1/2}u|_{\alpha,\varphi(t)}^2 \le \frac3{2\delta}|f(t)|_{\alpha-1/2,\sigma}^2\leq \frac{3C_1^2}{2\delta}  e^{-2\lambda t}.
\eeq
Applying Gronwall's inequality in \eqref{s2} yields for all $t\in(0,T)$ that
\begin{align*}
|u(t)|_{\alpha,\varphi(t)}^2
&\le e^{-2(1-\delta)t}|u^0|_{\alpha,0}^2+\frac{3C_1^2}{2\delta}e^{-2(1-\delta)t}\int_0^t  e^{2(1-\delta)\tau} \cdot e^{-2\lambda\tau} d\tau\\
&\le e^{-2(1-\delta)t}|u^0|_{\alpha,0}^2+\frac{3C_1^2}{4\delta (\lambda-1+\delta)} e^{-2(1-\delta)t}\\
&=  \Big(|u^0|_{\alpha,0}^2 + C_0^2\Big)e^{-2(1-\delta)t}.
\end{align*}

Combining this with condition \eqref{usmall} for the initial data, we obtain
\beqs
|u(t)|_{\alpha,\varphi(t)}^2
\le  2C_0^2e^{-2(1-\delta)t} ,
\eeqs
which gives
\beq\label{s4}
|u(t)|_{\alpha,\varphi(t)}
\le  \sqrt{2} C_0 e^{-(1-\delta)t}, \quad \forall t\in(0,T).
\eeq
In particular, letting $t\to T^-$ in \eqref{s4} yields
\beqs
\lim_{t\to T^-}|u(t)|_{\alpha,\varphi(t)}
\le \sqrt{2}C_0<2C_0.
\eeqs

By the standard contradiction argument, we have that the inequality in \eqref{uT} holds for all $t>0$, and consequently, so does \eqref{s4}.  Since $\varphi(t)=\sigma$ for $t\ge t_*$, the desired estimate \eqref{uest} follows \eqref{s4}.

For $t\ge t_*$, integrating \eqref{s2} from $t$ to $t+1$ gives
\begin{align*}
&  2(1-\delta)\int_t^{t+1} |A^{1/2}u(\tau)|_{\alpha,\sigma}^2d\tau
 \le |u(t)|_{\alpha,\sigma}^2+\frac{3C_1^2}{2\delta}\int_t^{t+1}e^{-2\lambda \tau}d\tau\\
&\le 2C_0^2 e^{-2(1-\delta)t}+\frac{3C_1^2}{4\delta\lambda}e^{-2\lambda t}
\le C_0^2 e^{-2(1-\delta)t}\Big(2+\frac{\lambda-1+\delta}{\lambda}\Big)\\
&\le 3C_0^2 e^{-2(1-\delta)t}.
\end{align*}
Then estimate \eqref{intAa} follows.

\textbf{Part II: case $\sigma=0$.}  The proof is similar to Part I without using the function $\varphi(t)$.
Here, we perform necessary calculations.
First, using Sobolev norms, we have
 \beq\label{dtAa}
 \frac12\ddt |A^\alpha u|^2+\Big(1-\frac\delta2-K^\alpha|A^\alpha u|\Big)|A^{\alpha+1/2}u|^2\le \frac1{2\delta}|A^{\alpha-1/2}f|^2.
 \eeq

 As long as $|A^\alpha u(t)|\le 2C_0=\delta/(2K^\alpha)$ in $[0,T)$ for some $T\in(0,\infty]$, we have 
 for $t\in(0,T)$ that
 \begin{align*}
 |A^\alpha u(t)|^2
 &\le  |A^\alpha u^0|^2 e^{-2(1-\delta)t} + \frac{C_1^2 e^{-2(1-\delta)t}}{\delta}
 \int_0^t e^{-2(\lambda-\delta+1)\tau}d\tau\\
 &\le  \Big(C_0^2 +\frac{C_1^2}{2\delta(\lambda-\delta+1)}\Big) e^{-2(1-\delta)t}= 2C_0^2 e^{-2(1-\delta)t}. 
 \end{align*}
This implies $T=\infty$ and then also proves \eqref{uest}.
Now, using \eqref{uest} for the second term on the left-hand side of \eqref{dtAa}, and then integrating in time gives
 \begin{align*}
 2(1-\delta)\int_t^{t+1} |A^{\alpha+1/2}u|^2d\tau
&\le |A^\alpha u(t)|^2+ \frac1\delta\int_t^{t+1}|A^{\alpha-1/2}f|^2d\tau\\
 &\le 2C_0^2 e^{-2(1-\delta)t} + \frac{C_1^2 e^{-2\lambda t}}{2\delta\lambda}
 \le  3C_0^2 e^{-2(1-\delta)t}. 
 \end{align*}
This implies \eqref{intAa}, and the proof  is complete.
\end{proof}

\begin{remark}
We point out that Proposition \ref{theo22} is essentially an application of Leray-Hopf energy inequality and the well-known fact (cf. \cite{DoeringTiti}) that having control on the growth of the rate of global energy dissipation, $\epsilon:=\sup_{t\geq0}\epsilon(t)$, where $\epsilon(t):=\nu\lVert{\nabla u(t)}\rVert_{L^2}^2$, can sustain exponential decay in the spectrum of the corresponding solution.  The need for proving Proposition \ref{theo22} comes from the fact that we require detailed information of the decay rates of the solution and the particular effect on it from the body force, which is not immediately available from \cite{DoeringTiti}, where the case of having a non-potential force is not treated.
\end{remark}

Considering decaying forces, we assume at the moment up to  Proposition \ref{theo23} that there are numbers $M_*,\kappa_0>0$ such that
\beq\label{fkappa}
|f(t)|\le M_* e^{-(1+\kappa_0)t/2},\quad \forall t\ge 0.
\eeq

We recall estimate (A.39) of \cite[Chap. II]{FMRTbook} for Leray-Hopf weak solutions (under the Basic Assumption,)
\beqs
|u(t)|^2\le e^{-t}|u_0|^2 + e^{-t}\int_0^t e^\tau |f(\tau)|^2d\tau,\quad \forall t>0.
\eeqs
It then follows from \eqref{fkappa} that
\beq\label{uenerM}
|u(t)|^2\le e^{-t}(|u_0|^2 + M_*^2/\kappa_0),\quad \forall t>0.
\eeq

By applying  the Cauchy-Schwarz, Cauchy,  and Poincar\'e inequalities to the last term on the right-hand side of \eqref{Lenergy}, upon simplifying we obtain
	\beqs
		|u(t)|^2+\int_{t_0}^t \|u(\tau)\|^2d\tau\le |u(t_0)|^2+\int_{t_0}^t |f(\tau)|^2\ d\tau,
	\eeqs
for $t_0=0$ and almost all $t_0\in(0,\infty)$, and all $t\ge t_0$.
Using \eqref{uenerM} for $|u(t_0)|^2$, and, again, \eqref{fkappa} yields
\beq\label{uVest}
\int_{t_0}^{t_0+1}\|u(\tau)\|^2d\tau 
\le e^{-t_0}\Big(|u_0|^2 +  \frac{M_*^2}{\kappa_0}\Big) + \frac1{1+\kappa_0}M_*^2 e^{-(1+\kappa_0)t_0}
\le e^{-t_0}\Big(|u_0|^2 + \frac{2 M_*^2}{\kappa_0}\Big).
\eeq

For any $t\ge 0$, let $\{t_n\}_{n=1}^\infty$ be a sequence in $(0,\infty)$ converging to $t$ such that \eqref{uVest} holds for $t_0=t_n$. Then letting $n\to\infty$ gives
\beq\label{tt1}
\int_t^{t+1}\|u(\tau)\|^2d\tau 
\le e^{-t}\Big(|u_0|^2 + \frac{2 M_*^2}{\kappa_0}\Big).
\eeq




\begin{proposition}\label{theo23}
Assume \eqref{fkappa} and, additionally, that there are $\sigma\ge 0$, $\alpha\ge 1/2$ and $\lambda_0\in(0,1)$ such that
\beq\label{falphaonly}
|f(t)|_{\alpha,\sigma}=\mathcal O(e^{-\lambda_0t})\quad\text{as }t\to\infty.
\eeq

Let $u(t)$ be a Leray-Hopf weak solution of \eqref{fctnse}.  
Then for any $\delta\in (1-\lambda_0,1)$, there exists  $T_*>0$ 
such that $u(t)$ is a regular solution of \eqref{fctnse} on $[T_*,\infty)$, and one has for all $t\ge 0$ that
 \beq\label{preus0}
 |u(T_*+t)|_{\alpha+1/2,\sigma} \le K^{-\alpha-1/2}e^{-(1-\delta)t},
 \eeq
 \beq\label{preBlt}
|B(u(T_*+t),u(T_*+t))|_{\alpha,\sigma}\le K^{-\alpha-1}e^{-2(1-\delta)t},
\eeq
where $K$ is the constant in Lemma \ref{nonLem}.
\end{proposition}
\begin{proof}
First, we note that \eqref{preBlt} is a direct consequence of \eqref{preus0}. Indeed, applying Lemma \ref{nonLem} with the use of \eqref{preus0}, we have  for $t\ge 0$,
 \beqs
|B(u(T_*+t),u(T_*+t))|_{\alpha,\sigma}\le K^\alpha |u(T_*+t))|_{\alpha+1/2,\sigma}^2\le K^\alpha \Big(\frac{e^{-(1-\delta)t}}{K^{\alpha+1/2}}\Big)^2,
\eeqs
which yields \eqref{preBlt}. 

We focus on proving \eqref{preus0} now.
Define 
\beqs
\lambda=\frac{1-\delta+\lambda_0}{2}\in(1-\delta,\lambda_0).
\eeqs

We consider each case $\sigma>0$ and $\sigma=0$ separately.

\textbf{(i) Case $\sigma>0$.}

\textit{Step 1.}
By \eqref{tt1} and \eqref{falphaonly}, there exists $t_0>0$ such that
\beqs
|A^{1/2}u(t_0)|<C_0(1/2,\delta),
\eeqs
\beqs
|f(t_0+t)|_{0,\sigma}\le C_1(1/2,\delta,\lambda)e^{-\lambda t},\quad \forall t\ge 0.
\eeqs

Applying Proposition \ref{theo22} to $u(t_0+\cdot)$,  $f(t_0+\cdot)$, $\alpha=1/2$  results in
\beqs
|u(t_0+t)|_{1/2,\sigma}\le \sqrt 2 C_0(1/2,\delta) e^{-(1-\delta)t}\le K^{-1/2}e^{-(1-\delta)t},\quad \forall t\ge t_*=6\sigma/\delta.
\eeqs

Then by \eqref{als}, we have for all $t\ge t_*$ that
\begin{align}
|A^{\alpha+1/2} u(t_0+t)|
&\le  \Big(\frac{2\alpha+1}{e\sigma}\Big)^{2\alpha+1}  |e^{\sigma A^{1/2}} u(t_0+t)|
\le \Big(\frac{2\alpha+1}{e\sigma}\Big)^{2\alpha+1}  |u(t_0+t)|_{1/2,\sigma} \notag \\
&\le \Big(\frac{2\alpha+1}{e\sigma}\Big)^{2\alpha+1} K^{-1/2} e^{-(1-\delta)t}.\label{Aaut0}
\end{align}

\textit{Step 2.}
From \eqref{Aaut0} and \eqref{falphaonly} we deduce that there is a sufficiently large  $T>t_0+t_*$ so that
\beqs
|A^{\alpha+1/2} u(T)|\le C_0(\alpha+1/2,\delta),
\eeqs
\beqs
|f(T+t)|_{\alpha,\sigma}\le C_1(\alpha+1/2,\delta,\lambda)e^{-\lambda t} \quad \forall t\ge 0.
\eeqs

Applying Proposition \ref{theo22} again to $u(T+\cdot)$,  $\alpha:=\alpha+1/2$, we obtain that there is $T_*>T+t_*$  such that
\beqs
|u(T_*+t)|_{\alpha+1/2,\sigma}\le \sqrt 2 C_0(\alpha+1/2,\delta)e^{-(1-\delta)t}\le \frac1{K^{\alpha+1/2}}e^{-(1-\delta)t}\quad \forall t\ge 0.
\eeqs 
This yields \eqref{preus0} and completes the proof of Case $(i)$.

\textbf{(ii) Case $\sigma=0$.}
We will apply Proposition \ref{theo22} recursively to gain the exponential decay for $u(t)$ in higher Sobolev norms.

For $j\in \N$, suppose
\beq\label{intAj}
\lim_{t\to\infty}\int_t^{t+1}|A^{j/2}u(\tau)|^2d\tau=0, 
\eeq
and
\beq\label{fj}
|A^{(j-1)/2}f(t)|=\mathcal O(e^{-\lambda_0 t})\quad\text{as }t\to\infty.
\eeq

Then there is $T>0$ so that
\beqs
|A^{j/2} u(T)|\le C_0(j/2,\delta),
\eeqs
\beqs
|A^{j/2-1/2}f(T+t)|\le C_1(j/2,\delta,\lambda)e^{-\lambda t} \quad \forall t\ge 0.
\eeqs

Applying Proposition \ref{theo22}  to $u(T+\cdot)$,  $\alpha:=j/2$, $\sigma:=0$, we obtain 
\beqs
|A^{j/2}u(T+t)|\le \sqrt 2 C_0(j/2,\delta)e^{-(1-\delta)t}\le \frac1{K^{j/2}}e^{-(1-\delta)t}\quad \forall t\geq0,
\eeqs 
and
\beq\label{intAj1}
\int_t^{t+1}|A^{(j+1)/2}u(\tau)|^2d\tau=\mathcal O(e^{-2(1-\delta)t})\text{ as }t\to\infty.
\eeq

Note, by \eqref{tt1}, that \eqref{intAj} holds true for $j=1$.
Let $m\in\N\cup\{0\}$ be given such that
\beq \label{mal}
\alpha \le m/2<\alpha+1/2.
\eeq

Since $\alpha\ge1/2$, condition \eqref{mal} gives $m\ge1$. Also, observe that \eqref{mal} implies $(m-1)/2<\alpha$. Hence, by \eqref{falphaonly}, condition \eqref{fj} is satisfied for $j=1,2,\ldots,m$.
Now we repeat the argument from \eqref{intAj} to \eqref{intAj1} for  $j=1,2,\ldots,m$.
Particularly, when $j=m$ we obtain from \eqref{intAj1} that 
\beqs
\int_t^{t+1}|A^{(m+1)/2}u(\tau)|^2d\tau=\mathcal O(e^{-2(1-\delta)t}) \text{ as }t\to\infty.
\eeqs
Since $\alpha\le m/2$, this yields
\beq\label{intAm}
\int_t^{t+1}|A^{\alpha+1/2}u(\tau)|^2d\tau=\mathcal O(e^{-2(1-\delta)t}) \text{ as }t\to\infty.
\eeq

Using \eqref{intAm} in place of \eqref{Aaut0}, we can proceed as in Step 2 of part (i)  and obtain \eqref{preus0}. The proof is complete.
\end{proof}

\section{Proofs of main results}\label{pfsec}

We will use the following elementary identities: for $\beta>0$, integer $d\ge 0$, and any $t\in \R$,
\beq\label{id1}
\int_{-\infty}^t  \tau^d e^{\beta \tau}\ d\tau=\frac{e^{\beta t}}{\beta}\sum_{n=0}^{d}\frac{(-1)^{d-n}  d!}{n!\beta^{d-n}}t^n,
\eeq
\beq\label{id2}
\int_t^\infty  \tau^d e^{-\beta \tau}\ d\tau=\frac{e^{-\beta t}}{\beta}\sum_{n=0}^{d}\frac{d!}{n!\beta^{d-n}}t^n.
\eeq

The next lemma is a building block of the construction of the polynomials $q_n(t)$'s. It summarizes and reformulates the facts
used in \cite{FS87} and \cite[Lemma 3.2]{FS91}, see also \cite{HM1}. 

\begin{definition}
 Let $X$ be a Banach space with its dual $X'$. Let $u(t)$ and $g(t)$ be functions in  $L^1_{\rm loc}([0,\infty),X)$.
 We say $g(t)$ is the  $X$-valued distribution derivative of $u(t)$, and denote $g=u'$, if
 \beq\label{disprimet}
 \ddt \inprod{u(t),v}=\inprod{g(t),v}\text{ in the distribution sense on $(0,\infty)$}, \forall  v\in X',
 \eeq
 where $\inprod{\cdotp,\cdotp}$ in \eqref{disprimet} denotes the usual duality pairing between an element of $X$ and $X'$.
\end{definition}

\begin{lemma}\label{polylem}
 Let $(X,\|\cdot\|)$ be a Banach space. Suppose $y(t)$ is a function in $C([0,\infty),X)$ that solves the following ODE 
 \beqs
 y'(t)+ \beta y(t) =p(t)+g(t)
 \eeqs
 in the $X$-valued distribution sense on $(0,\infty)$.
Here,  $\beta\in \R$ is a fixed constant, $p(t)$ is an $X$-valued polynomial in $t$, and $g\in L^1_{\rm loc}([0,\infty),X)$ satisfies
 \beqs
 \|g(t)\|\le Me^{-\delta t} \quad \forall t\ge 0, \quad \text{for some } M,\delta>0.
 \eeqs

 Define $q(t)$ for $t\in  \R$ by 
\beq\label{qdef}        
 q(t)=
\begin{cases}
e^{-\beta t}\int_{-\infty}^t e^{\beta\tau }p(\tau) d\tau&\text{if }\beta >0,\\
y(0) +\int_0^\infty  g(\tau)d\tau + \int_0^t p(\tau)d\tau &\text{if }\beta =0,\\
-e^{-\beta t}\int_t^\infty e^{\beta\tau }p(\tau) d\tau&\text{if }\beta <0. 
\end{cases}
\eeq

Then $q(t)$ is an $X$-valued polynomial of degree at most $\deg(p)+1$ that  satisfies 
\beq\label{pode2}
q'(t)+\beta q(t) = p(t),\quad t\in \R,
\eeq
and the following estimates hold:

\begin{enumerate}[label={\rm (\roman*)}]
  \item 
  If $\beta>0$ then
  \beq\label{g1b1}
  \|y(t)-q(t)\|\le \Big(\|y(0)-q(0)\| + \frac{M}{|\beta-\delta|}\Big)e^{-\min\{\delta,\beta\} t}, \quad t\ge 0,\text{ for } \beta\ne \delta,
  \eeq
  and
  \beq\label{g1b1equal}
  \|y(t)-q(t)\|\le (\|y(0)-q(0)\| + M t )e^{-\delta t}, \quad t\ge 0,\text{ for } \beta=\delta.
  \eeq
  
  \item If $\beta=0$ then 
  \beq\label{g1b2}
  \|y(t)-q(t)\|\le \frac{M}{\delta}e^{-\delta t},\quad t\ge 0.
  \eeq

  \item If $\beta<0$ and 
  \beq\label{yexpdec}
  \lim_{t\to\infty} (e^{\beta t}y(t))=0,
  \eeq
  then
  \beq\label{g1b3}
  \|y(t)-q(t)\|\le \frac{M}{|\beta|+\delta}e^{-\delta t}, \quad t\ge 0.
  \eeq
  \end{enumerate}
\end{lemma}
\begin{proof}
The fact that $q(t)$ is a polynomial in $t$ follows the identities \eqref{id1} and\eqref{id2}.
The equation \eqref{pode2} obviously results  from the definition \eqref{qdef} of $q(t)$. It remains to prove estimates \eqref{g1b1}, \eqref{g1b2} and \eqref{g1b3}.

Let $z(t)=y(t)-q(t)$, then
\beqs
z'(t)+\beta z(t)=g(t)\text{ in the $X$-valued distribution sense on } (0,\infty).
\eeqs

Multiplying this equation  by $e^{\beta t}$ yields
\beq\label{eaz}
(e^{\beta t} z(t))'=e^{\beta t} g(t) \text{ in the $X$-valued distribution sense on } (0,\infty).
\eeq

For $t_0\ge 0$, it follows \eqref{eaz} and \cite[Ch. III, Lemma 1.1]{TemamAMSbook} 
that
\beq\label{prevoc}
e^{\beta t} z(t)=\xi+\int_{t_0}^t e^{\beta \tau}g(\tau)d\tau,
\eeq 
for some $\xi\in X$ and almost all $t\in(t_0,\infty)$.

Since $e^{\beta t} z(t)$ is continuous on $[0,\infty)$ and $e^{\beta t} g(t)\in L^1_{\rm loc}([0,\infty))$, we have $\xi=e^{\beta t_0} z(t_0)$ and  equation \eqref{prevoc} holds for all $t\ge t_0$. Hence, we obtain the standard variation of constant formula
\beq\label{voc}
z(t)=e^{-\beta(t-t_0)} z(t_0)+e^{-\beta t} \int_{t_0}^t e^{\beta\tau}g(\tau)d\tau \quad \forall t\ge t_0.
\eeq

(i) Case $\beta>0$.
Setting $t_0=0$ in \eqref{voc},  we estimate   
\begin{align*}
\|z(t)\|
&\le e^{-\beta t}\|z(0)\|+ e^{-\beta t} \int_0^t e^{\beta \tau}\|g(\tau)\| d\tau
 \le e^{-\beta t}\|z(0)\| +   e^{-\beta t} \int_0^t e^{\beta\tau} Me^{-\delta\tau} d\tau. 
\end{align*}

Since the last term is $M(e^{-\delta t}-e^{-\beta t})/(\beta-\delta)$ if $\beta\ne \delta$, and is $M t e^{-\delta t}$ if $\beta=\delta$, we easily obtain  \eqref{g1b1} and \eqref{g1b1equal}.

(ii) Case $\beta=0$. Note from \eqref{qdef} that $z(0)=y(0)-q(0)=-\int_0^\infty g(\tau)d\tau$. 
Letting $t_0=0$ in \eqref{prevoc} gives  
\beqs
z(t)=z(0) +\int_0^t  g(\tau)d\tau = - \int_t^\infty  g(\tau)d\tau.
\eeqs
Hence
\beqs
\|z(t)\|\le \int_t^\infty \|g(\tau)\|d\tau \le \int_t^\infty Me^{-\delta \tau}d\tau=\frac{M}\delta e^{-\delta t},
\eeqs
which proves \eqref{g1b2}.

(iii) Case $\beta<0$. By \eqref{yexpdec} and the fact $q(t)$ is a polynomial, we have $e^{\beta t}z(t)\to 0$ as $t\to\infty$.
Then letting $t\to\infty$ in \eqref{prevoc} and setting $t_0=t$ yield
\beqs
 z(t)=-e^{-\beta t} \int_t^\infty e^{\beta\tau} g(\tau) d\tau.
\eeqs
It follows that 
\beqs
\|z(t)\|\le e^{-\beta t} \int_t^\infty M e^{(\beta-\delta)\tau} d\tau = e^{-\beta t}  \frac{Me^{(\beta-\delta)t}}{-\beta+\delta}
=\frac{M}{|\beta|+\delta}e^{-\delta t},
\eeqs
which proves \eqref{g1b3}.
\end{proof}

The remainder of this paper is focused on the proofs of main results, and will use the following notation.

\begin{notation} If $n\in\sigma (A)$, we define  $R_n$ to be the orthogonal projection in $H$ on the eigenspace of $A$ corresponding to $n$. In case $n\notin \sigma (A)$,  set  $R_n=0$.  

For $n\in \N$, define $P_n=R_1+R_2+ \cdots +R_n$. Note that each vector space $P_nH$ is finite dimensional.
\end{notation}

\subsection{Proof of Theorem \ref{mainthm}}

We start by obtaining some additional properties for the force $f(t)$ and solution $u(t)$ which we will make use of later.

By the expansion \eqref{forcexpand} of $f(t)$ in  $E^{\infty,\sigma_0}$, for each $N\in\N$ and $\alpha\geq0$, there exists a number $\delta_{N,\alpha}\in(0,1)$ such that
\beq\label{Fcond}
\Big|f(t)-\sum_{n=1}^N f_n(t)e^{-nt}\Big|_{\alpha,\sigma_0}=\mathcal O(e^{-(N+\delta_{N,\alpha})t})\quad \text {as }t\to\infty.
\eeq

Observe that we have the following immediate consequences:
\begin{enumerate}[label={\rm (\alph*)}]
\item The relation \eqref{Fcond} implies for each $\alpha\ge 0$ that  $f(t)$ belongs to $G_{\alpha,\sigma_0}$ for $t$ large.
\item Note that when $N=1$, the function $f(t)$ itself satisfies
\beqs
|f(t)- f_1(t) e^{-t}|_{\alpha,\sigma_0}=\mathcal O(e^{-(1+\delta_{1,\alpha}) t} ).
\eeqs

Since $f_1(t)$ is a polynomial, it follows  that
\beq\label{fsure}
|f(t)|_{\alpha,\sigma_0}=\mathcal O(e^{-\lambda t}),\quad \forall \lambda\in (0,1),\quad\forall \alpha\ge 0.
\eeq

Consequently, for any $\varep>0$, $\alpha\ge 0$,  and $\lambda\in(0,1)$, applying \eqref{fsure} with $(\lambda+1)/2$ replacing $\lambda$, it follows that there is $T>0$ such that 
\beq\label{fdecay}
|f(T+t)|_{\alpha,\sigma_0}\le \varep e^{-\lambda t}\quad \forall t\ge 0.
\eeq

\item Combining \eqref{fsure} for $\alpha=0$, with the Basic Assumption, we assume, without loss of generality,  for each $\lambda\in (0,1)$ that 
\beq\label{fh}
|f(t)|\le M_\lambda e^{-\lambda t},\quad \forall t\ge 0, \text{ for some }M_\lambda>0.
\eeq
\end{enumerate}

For the solution $u(t)$, we summarize the key estimates in section \ref{Gdecay} into the following.

\textbf{Claim.} For any $\alpha\geq0$ and $\delta\in (0,1)$, there exists a positive number $T_*>0$ 
such that $u(t)$ is a regular solution on $[T_*,\infty)$, and one has for $t\ge 0$ that
 \beq\label{us0}
 |u(T_*+t)|_{\alpha+1/2,\sigma_0} \le e^{-(1-\delta) t},
 \eeq
 \beq\label{Blt}
|B(u(T_*+t),u(T_*+t))|_{\alpha,\sigma_0}\le e^{-2(1-\delta) t}.
\eeq

\textit{Proof of {Claim}.}
We apply Proposition \ref{theo23}. By \eqref{fdecay}, we have that \eqref{falphaonly} holds for $\sigma=\sigma_0$, any $\alpha\ge 1/2$ and any $\lambda_0\in(0,1)$. 
 Also, \eqref{fkappa} is satisfied because of \eqref{fh}.
 Therefore \eqref{preus0} and \eqref{preBlt} hold for $\sigma=\sigma_0$, any $\alpha\ge 1/2$ and any $\delta\in(0,1)$. These directly yield \eqref{us0} and \eqref{Blt}.
 
 \bigskip
Returning to the main proof, it suffices to prove that there exist polynomials $q_n$'s for all $n\ge 1$ such that for each $N\ge 1$ the following properties ($\mathcal H1$), ($\mathcal H2$), and ($\mathcal H3$) hold true:
\begin{enumerate}
 \item[($\mathcal H1$)]  $q_N\in \mathcal P^{\infty,\sigma_0}$.
 \item[($\mathcal H2$)] For $\alpha\ge 1/2$,
 \beq\label{remdelta}
\Big |u(t)-\sum_{n=1}^N q_n(t) e^{-nt}\Big |_{\alpha,\sigma_0} =\mathcal O(e^{-(N+\varep)t} )\quad\text{as } t\to\infty,\quad\forall\varep\in(0,\delta_{N,\alpha}^*),
 \eeq
where the numbers $\delta_{n,\alpha}^*$'s, for $\alpha\ge 1/2$, are defined recursively by
\beqs
 \delta_{n,\alpha}^*=\begin{cases}
	       \delta_{1,\alpha},&\text{for }n=1,\\
             \min\{\delta_{n,\alpha},\delta^*_{n-1,\alpha+1/2}\},&\text{for }n\ge 2.
            \end{cases}
\eeqs

 \item[($\mathcal H3$)] The ODE  \eqref{unODE} holds in $E^{\infty,\sigma_0}$ for $n=N$.
 \end{enumerate} 

\bigskip
We prove these statements by constructing the polynomials $q_N(t)$'s recursively.

\medskip

\noindent {\bf Base case: $N=1$.} 
Let $k\geq1$.  By taking $v\in R_kH$ in the the weak formulation \eqref{varform}, we have
\beq\label{rku}
\ddt R_ku + k R_k u = R_k (f(t) - B(u(t),u(t)))
\eeq
in the $R_kH$-valued distribution sense on $(0,\infty)$. 
(Since $R_kH$ is finite dimensional, ``$R_kH$-valued distribution sense," is simply the same as ``distribution sense''.)

Let $w_0(t)=e^t u(t)$ and $w_{0,k}(t)=R_k w_0(t)$. By virtue of the $H_{\rm w}$-continuity of $u(t)$ (see \eqref{lh:wksol}), we have $w_{0,k}\in C([0,\infty),R_kH)$. It follows from \eqref{rku} that
\beq\label{wj}
\ddt w_{0,k} + (k-1) w_{0,k} = R_k f_1 + R_kH_0(t),
\eeq
where
\beq\label{H0def}
H_0(t)=e^t (f(t)-F_1(t) - B(u(t),u(t))),
\eeq
and $F_1$ is defined in \eqref{uF}.  Note that $R_k f_1(t)$ is an $R_kH$-valued polynomial in $t$.

Let $\alpha\ge 1/2$ be fixed. 
Using \eqref{Fcond} for $N=1$ and applying \eqref{Blt} with $\delta=(1-\delta_{1,\alpha})/4$, there are $T_0>0$ and $D_0\ge 1$ such that for $t\ge 0$, 
\beq\label{ch1}
e^t |f(T_0+t)-F_1(T_0+t)|_{\alpha,\sigma_0}\le D_0e^{-\delta_{1,\alpha} t}, 
\eeq
\beq\label{ch2}
e^t |B(u(T_0+t),u(T_0+t))|_{\alpha,\sigma_0}\le e^{(2\delta-1)t}=e^{-(1+\delta_{1,\alpha})t/2}\le e^{-\delta_{1,\alpha} t}.
\eeq
Then, by setting $D_1=D_0+1$,  we have
\beq\label{ch3}
|H_0(T_0+t)|_{\alpha,\sigma_0} \le D_1e^{-\delta_{1,\alpha} t},\quad\forall t\ge0.
\eeq

We will now identify the components of the desired polynomial, $q_1(t)$, belonging to each eigenspace $R_kH$.

\medskip
\noindent\textit{\underline{Case: $k=1$}.}  
Applying Lemma \ref{polylem}(ii) to equation \eqref{wj} with $X=R_1H$, $\|\cdot\|=|\cdot|_{\alpha,\sigma_0}$, $\beta=0$,
$$y(t)=w_{0,1}(T_0+t),\quad p(t)=R_1f_1(T_0+t), \quad g(t)=R_1H_0(T_0+t),$$
we infer that there is an $R_1H$-valued polynomial $q_{1,1}(t)$ such that for any $t\ge 0$
\beqs
|w_{0,1}(T_0+t)-q_{1,1}(t)|_{\alpha,\sigma_0} \le \frac{D_1}{\delta_{1,\alpha}}e^{-\delta_{1,\alpha} t},
\eeqs
thus,
\beq\label{q11}
|R_1 w_0(t)-q_{1,1}(t-T_0)|_{\alpha,\sigma_0} \le \frac{D_1e^{\delta_{1,\alpha} T_0}}{\delta_{1,\alpha}} e^{-\delta_{1,\alpha} t},\quad\forall t\geq T_0.
\eeq

In fact,
\beq\label{q11def}
q_{1,1}(t)= \xi_1+\int_0^t R_1f_1(\tau+T_0)d\tau\quad \text{for some } \xi_1\in R_1H.
\eeq

\medskip
\noindent\textit{\underline{Case: $k\geq2$}.}  We apply Lemma \ref{polylem}(i) to equation \eqref{wj} with 
$$y(t)=w_{0,k}(T_0+t),\quad p(t)=R_kf_1(T_0+t),\quad  g(t)=R_kH_0(T_0+t),$$
where $\beta=k-1>\delta_{1,\alpha}$, and the norm $\|\cdot\|$ being $|\cdot|_{\alpha,\sigma_0}$ on the space $X=R_kH$.
In particular, there is an $R_kH$-valued polynomial, $q_{1,k}(t)$ such that for any $t\ge 0$
\beqs
|w_{0,k}(T_0+t)- q_{1,k}(t)|_{\alpha,\sigma_0} \le e^{-\delta_{1,\alpha} t}\Big(|w_{0,k}(T_0)|_{\alpha,\sigma_0} +|q_{1,k}(0)|_{\alpha,\sigma_0} + \frac {D_1}{k-1-\delta_{1,\alpha}}\Big),
\eeqs
which implies for all $t\ge T_0$ that
\beq\label{q1j}
|w_{0,k}(t)- q_{1,k}(t-T_0)|_{\alpha,\sigma_0} \le e^{-\delta_{1,\alpha} (t-T_0)}\Big(|w_k(T_0)|_{\alpha,\sigma_0} +|q_{1,k}(0)|_{\alpha,\sigma_0} + \frac {D_1}{k-1-\delta_{1,\alpha}}\Big).
\eeq

In fact, for $k\ge 2$
\beq\label{q1k}
q_{1,k}(t)=-e^{(1-k)t}\int_{-\infty}^t e^{(k-1)\tau}R_kf_1(T_0+\tau)d\tau.
\eeq

\noindent{\bf Polynomial $q_1(t)$.} 
Define
\beq\label{q1def}
q_1(t)=\sum_{k=1}^\infty q_{1,k}(t-T_0),\quad t\in\R.
\eeq

We next prove that $q_1\in\mathcal{P}^{\infty,\sigma_0}$. Write
\beqs
f_1(t+T_0)=\sum_{d=0}^m a_d t^d,\quad \text{for some }a_d\in E^{\infty,\sigma_0}.
\eeqs

Clearly, by \eqref{q11def}, $R_1q_1(t+T_0)=q_{1,1}(t)$ is a  $\mathcal V$-valued polynomial, and hence, 
\beq\label{R1q1poly}
\text{the mapping }t\mapsto R_1q_1(t+T_0) \text{ belongs to } \mathcal{P}^{\infty,\sigma_0}.
\eeq

We consider the remaining part $(I-R_1)q_1(t+T_0)$.
Using the integral formula \eqref{id1}, 
\begin{align*}
(I-R_1)q_1(t+T_0)
&=\sum_{k=2}^\infty q_{1,k}(t)
=\sum_{k=2}^\infty -e^{(1-k)t} \int_{-\infty}^t e^{(k-1)\tau}\Big(\sum_{d=0}^m R_k a_d  \tau^d\Big) d\tau\\
&=\sum_{k=2}^\infty \frac{1}{k-1}\sum_{d=0}^m \sum_{n=0}^{d}\frac{(-1)^{d-n}  d!}{n!(k-1)^{d-n}}t^n R_k a_d\\
&=\sum_{k=2}^\infty \frac{1}{k-1}\sum_{n=0}^d\Big(\sum_{d=n}^{m}\frac{(-1)^{d-n}  d!}{n!(k-1)^{d-n}}R_k a_d\Big)t^n.
\end{align*}
Thus, 
\beq\label{IRq1}
(I-R_1)q_1(t+T_0)=\sum_{n=0}^d b_nt^n,
\eeq
where
the coefficient $b_n$, for $0\le n\le d$,  is 
\begin{align*}
b_n=\sum_{k=2}^\infty \frac{1}{k-1}\Big(\sum_{d=n}^{m}\frac{(-1)^{d-n}  d!}{n!(k-1)^{d-n}}R_k a_d\Big).
\end{align*}

For any $\mu\ge 0$, we have
\begin{align*}
|b_n|_{\mu+1,\sigma_0}^2=|Ab_n|_{\mu,\sigma_0}^2
&=\sum_{k=2}^\infty  \Big|\frac{1}{k-1}\sum_{d=n}^m \frac{(-1)^{d-n}  d!}{n!(k-1)^{d-n}} k \cdot R_k a_d\Big |_{\mu,\sigma_0}^2\\
&\le \sum_{k=2}^\infty \frac{k^2}{(k-1)^2} \Big(\sum_{d=n}^m \frac{ m!}{n!} |R_k a_d|_{\mu,\sigma_0}\Big)^2\\
&\le 4\sum_{k=2}^\infty  \frac{ (m!)^2  (m-n+1)^2 }{(n!)^2} |R_ka_d|_{\mu,\sigma_0}^2.
\end{align*}
(Above, we simply used $k/(k-1)\le 2$ for the last inequality.)
Thus,
\beq\label{Abn}
|b_n|_{\mu+1,\sigma_0}^2
\le \frac{ 4(m!)^2  (m-n+1)^2 }{(n!)^2} |(I-R_1)a_d|_{\mu,\sigma_0}^2<\infty.
\eeq

Therefore, $b_n\in E^{\infty,\sigma_0}$ for all $0\le n\le d$,
and, by \eqref{IRq1}, the mapping $t\mapsto (I-R_1)q_1(t+T_0)$ belongs to $\mathcal{P}^{\infty,\sigma_0}$.
This, together with  \eqref{R1q1poly}, implies that $t\mapsto q_1(t+T_0)$ belongs to $\mathcal{P}^{\infty,\sigma_0}$ and, ultimately, that $t\mapsto q_1(t)$ belongs to $\mathcal{P}^{\infty,\sigma_0}$, as desired.

\noindent{\bf Remainder estimate.} We estimate $|u(t)-q_1(t)e^{-t}|_{\alpha,\sigma_0}$ now.
Firstly, inequality \eqref{q11} yields
\beq\label{R1wq1}
|R_1(w_0(t)-q_1(t))|_{\alpha,\sigma_0}=\mathcal O( e^{-\delta_{1,\alpha} t}).
\eeq
Secondly, we have from \eqref{q1j} that
\begin{align*}
&\sum_{k= 2}^\infty |R_k(w_0(t)-q_1(t))|_{\alpha,\sigma_0}^2 \\
&\le 3 e^{2\delta_{1,\alpha} T_0} e^{-2\delta_{1,\alpha} t}\sum_{k=2}^\infty\Big (|w_{0,k}(T_0)|_{\alpha,\sigma_0}^2 +|R_kq_1(T_0)|_{\alpha,\sigma_0}^2+\frac {D_1^2}{(k-1-\delta_{1,\alpha})^2}\Big)\\
&\le D_2^2 e^{-2\delta_{1,\alpha} t}
\end{align*}
for all $t\ge T_0$, where
\beqs
D_2^2 =3 e^{2\delta_{1,\alpha} T_0}\Big\{|(I-R_1)w_0(T_0)|_{\alpha,\sigma_0}^2 + |(I-R_1)q_1(T_0)|_{\alpha,\sigma_0}^2+D_1^2\sum_{k=2}^\infty \frac {1 }{(k-1-\delta_{1,\alpha})^2})\Big\}<\infty.
\eeqs
This implies
\beq\label{wq1}
|(I-R_1)(w_0(t)-q_1(t))|_{\alpha,\sigma_0}\le D_2 e^{-\delta_{1,\alpha} t},\quad  \forall t\ge T_0.
\eeq

Combining \eqref{R1wq1} with \eqref{wq1}  gives
\beqs
|w_0(t)-q_1 (t)|_{\alpha,\sigma_0}= \mathcal O (e^{-\delta_{1,\alpha} t}),
\eeqs
and consequently,
\beq\label{rem1}
|u(t)-q_1 (t) e^{-t}|_{\alpha,\sigma_0}= \mathcal O (e^{-(1+\delta_{1,\alpha})t}).
\eeq
Thanks to \eqref{rem1}, the polynomial $q_1(t)$ is independent of $\alpha$. Hence the same $q_1$ satisfies \eqref{rem1} for all $\alpha\ge 1/2$, which proves ($\mathcal H2$) for $N=1$.

This proves \eqref{remdelta} for $N=1$.

\noindent{\bf Establishing the ODE \eqref{u1ODE}.} 
By \eqref{pode2} in Lemma \ref{polylem}, the polynomial $q_1(t)$ satisfies
\beq\label{pode3}
\ddt R_k q_1(T_0+t)+(k-1)R_kq_1(T_0+t)=R_kf_1(T_0+t),\quad \forall k\ge 1,\quad \forall t\in\R.
\eeq

For each $\mu\ge 0$, we have $Aq_1(T_0+t)$ and $f_1(T_0+t)$ belong to $G_{\mu,\sigma_0}$. Hence, 
we can sum over $k$ in \eqref{pode3} and  obtain
\beqs
\ddt q_1(t)+(A-1)q_1(t)=f_1(t) \quad\text{ in } G_{\mu,\sigma_0},\quad \forall t\in\R,
\eeqs
which implies that the differential equation \eqref{u1ODE} holds in $E^{\infty,\sigma_0}$.

Therefore, $q_1$ satisfies  ($\mathcal H1$), ($\mathcal H2$), and ($\mathcal H3$)  for $N=1$.

\medskip
\noindent{\bf Recursive step.} Let $N\ge1$. Suppose that there already exist $q_1,q_2,\ldots,q_N\in \mathcal P^{\infty,\sigma_0}$ that satisfy  ($\mathcal H2$),
and the ODE \eqref{unODE} holds in $E^{\infty,\sigma_0}$  for each $n=1,2,\ldots,N$.

We will construct a polynomial $q_{N+1}(t)$ that satisfies  ($\mathcal H1$), ($\mathcal H2$), and ($\mathcal H3$)  with $N+1$ replacing $N$.

Let $\alpha\ge 1/2$ be given and $\varep_*$ be arbitrary in $(0,\delta_{N+1,\alpha}^*)$.  
Define 
$$\bar u_N=\sum_{n=1}^N u_n\quad\text{and}\quad v_N=u-\bar u_N.$$
Assumption  ($\mathcal H2$) particularly yields 
\beq\label{vNrate}
| v_N(t)|_{\alpha+1/2,\sigma_0}=\mathcal O(e^{-(N+\varep)t}),\quad \forall \varep\in(0,\delta_{N,\alpha+1/2}^*).
\eeq

Subtracting \eqref{unODE} for $n=1,2,\ldots,N$ from \eqref{fctnse}, we have
\beq\label{uminusuN}
\frac d{dt}v_N+Av_N +B(u,u)- \sum_{m+j\le N} B(u_m,u_j)=f-\sum_{n=1}^N F_n,
\eeq
where the functions $F_n$'s are defined in \eqref{uF}.
We reformulate equation \eqref{uminusuN} as
\beq\label{vNeq}   
\frac d{dt}v_N+Av_N +\sum_{m+j=N+1} B(u_m,u_j)=F_{N+1}+h_N,
\eeq
where 
\begin{align*}
h_N&=-B(u,u)+\sum_{\stackrel{1\le m,j\le N}{m+j\le  N+1}} B(u_m,u_j)+f-\sum_{n=1}^{N+1} F_n\\
&=-\Big\{B(u,u)-B(\bar u_N,\bar u_N)\Big\}- \Big\{B(\bar u_N,\bar u_N)-\sum_{\stackrel{1\le m,j\le N}{m+j\le  N+1}} B(u_m,u_j)\Big\}+\Big\{f-\sum_{n=1}^{N+1} F_n\Big\}.
\end{align*} 
With this way of grouping, we rewrite $h_N$ as
\beq
\label{hdef}
h_N=-B(v_N,u)-B(\bar u_N,v_N)-\sum_{\stackrel{1\le m,j\le N}{m+j\ge  N+2}} B(u_m,u_j)+\tilde F_{N+1},
\eeq 
where
\beqs
\tilde F_{N+1}(t)=f(t)-\sum_{n=1}^{N+1} F_n(t).
\eeqs
Note in case $N=1$ that neither of the following terms 
$$ \sum_{m+j\le N} B(u_m,u_j)\text{ in \eqref{uminusuN}
 nor }\sum_{\stackrel{1\le m,j\le N}{m+j\ge  N+2}} B(u_m,u_j)\text{ in \eqref{hdef} will appear.}$$

\noindent{\bf Estimate of $h_N(t)$.}
By \eqref{Fcond} and Remark \ref{betterem}, we have
\beq \label{Ftil}
|\tilde F_{N+1}(t)|_{\alpha,\sigma_0}=\mathcal O(e^{-(N+1+\delta_{N+1,\alpha})t})
=\mathcal O(e^{-(N+1+\varep_*)t}).
\eeq

 It is also obvious that
\beq \label{Bmj}
\sum_{\stackrel{1\le m,j\le N}{m+j\ge  N+2}} |B(u_m,u_j)|_{\alpha,\sigma_0}
= \sum_{\stackrel{1\le m,j\le N}{m+j\ge  N+2}} e^{-(m+j)t}|B(q_m,q_j)|_{\alpha,\sigma_0} 
=\mathcal O(e^{-(N+1+\varep_*)t}).
\eeq

Take $\varep\in(\varep_*,\delta_{N+1,\alpha}^*)\subset (0,\delta_{N,\alpha+1/2}^*)$ in \eqref{vNrate}, and set $\delta=\varep-\varep_*\in(0,1)$.
Then we have from the definition of $u_n(t)$ and \eqref{us0} that 
\beq \label{ubu}
|\bar u_N(t)|_{\alpha+1/2,\sigma_0},|u(t)|_{\alpha+1/2,\sigma_0}=\mathcal O(e^{-(1-\delta)t}).
\eeq

By Lemma \ref{nonLem} and estimates \eqref{vNrate}, \eqref{ubu}, it follows that
\beq \label{BvNu}
|B(v_N,u)|_{\alpha,\sigma_0},|B(\bar u_N,v_N)|_{\alpha,\sigma_0}=\mathcal O(e^{-(N+\varep+1-\delta)t})=\mathcal O(e^{-(N+1+\varep_*)t}).\eeq

Therefore, by \eqref{hdef}, \eqref{Ftil}, \eqref{Bmj} and \eqref{BvNu},
\beq\label{hNo}
|h_N(t)|_{\alpha,\sigma_0}=\mathcal O(e^{-(N+1+\varep_*)t}).
\eeq

{\flushleft{\bf Construction of  $q_{N+1}(t)$.}}
Using the weak formulation of \eqref{vNeq}, which is similar to \eqref{varform}, and then taking the test function, $v$, to be in $R_kH$, we obtain 
\beq\label{vNeq1}   
\frac d{dt}R_kv_N+kR_k v_N +\sum_{m+j=N+1} R_k B(u_m,u_j)=R_k F_{N+1}+ R_k h_N \text{ on } (0,\infty).
\eeq

Let $w_{N}(t)=e^{(N+1)t} v_N(t)$ and $w_{N,k}=R_k w_N(t)$.
Using \eqref{vNeq1}, we write the ODE for $w_{N,k}$ as
\beq\label{wNeq}   
\frac d{dt}w_{N,k}+(k-(N+1))w_{N,k}=\Big(R_k{f}_{N+1}-\sum_{m+j=N+1} R_k B(q_m,q_j)\Big)+ H_{N,k},
\eeq
with
$H_{N,k}(t)= e^{(N+1)t}R_k h_N(t)$. 
Note from \eqref{hNo} that
\beqs
|H_{N,k}(t)|_{\alpha,\sigma_0}=\mathcal O(e^{-\varep_* t}).
\eeqs
Then there exist $T_N>0$ and $D_3>0$ such that
\beq\label{HNkD3}
|H_{N,k}(T_N+t)|_{\alpha,\sigma_0}\le D_3 e^{-\varep_* t},\quad \forall t\geq0.
\eeq

By the first property in \eqref{lh:wksol}, each $w_{N,k}(t)$ is continuous from $[0,\infty)$ to $R_kH$.

We apply Lemma \ref{polylem} to equation \eqref{wNeq} on $(T_N,\infty)$ with space $X=R_kH$,  norm $\|\cdot\|=|\cdot|_{\alpha,\sigma_0}$,
solution $y(t)=w_{N,k}(T_N+t)$, constant $\beta=k-(N+1)$, polynomial 
\begin{align*}
 p(t)&=R_k{f}_{N+1}(T_N+t)-\sum_{m+j=N+1} R_k B(q_m(T_N+t),q_j(T_N+t)),
 \end{align*}
and function $g(t)=H_{N,k}(T_N+t)$ which satisfies the estimate \eqref{HNkD3}.

\medskip
\noindent\textit{\underline{Case $k=N+1$}.}
We have $\beta=0$, then Lemma \ref{polylem}(ii) implies that there is a polynomial  $q_{N+1,N+1}(t)$ valued in $R_{N+1}H$ such that 
\beqs
|w_{N,N+1}(T_N+t)- q_{N+1,N+1}(t)|_{\alpha,\sigma_0}= \mathcal O (e^{-\varep_* t}).
\eeqs
Thus,
\beq \label{form0}
|R_{N+1}w_{N}(t)- q_{N+1,N+1}(t-T_N)|_{\alpha,\sigma_0}= \mathcal O (e^{-\varep_* t}).
\eeq

\medskip
\noindent\textit{\underline{Case $k\le N$}.} Note that $\beta<0$ and by \eqref{vNrate} we have
$$\lim_{t\to\infty}( e^{\beta t}y(t) ) = \lim_{t\to\infty}e^{\beta (t-T_N)} w_{N,k}(t)= e^{-\beta T_N}\lim_{t\to\infty}e^{kt}R_kv_N(t)=0.$$

Then, by applying Lemma \ref{polylem}(iii),  there is a polynomial  $q_{N+1,k}(t)$ valued in $R_kH$ such that 
\beqs
|w_{N,k}(T_N+t)- q_{N+1,k}(t)|_{\alpha,\sigma_0}=\mathcal O(e^{-\varep_* t}),
\eeqs
hence,
\beq\label{form1}
|R_k w_{N}(t)- q_{N+1,k}(t-T_N)|_{\alpha,\sigma_0}=\mathcal O(e^{-\varep_* t}). 
\eeq

\medskip
\noindent\textit{\underline{Case $k\ge N+2$}.}
Similarly, $\beta=k-N-1>\varep_*$ and, by  using Lemma \ref{polylem}(i), there is a polynomial  $q_{N+1,k}(t)$ valued in $R_kH$ such that for $t\ge 0$,
\beqs
|w_{N,k}(T_N+t)
- q_{N+1,k}(t)|_{\alpha,\sigma_0} 
\le \Big(|R_k v_N(T_N)|_{\alpha,\sigma_0} + |q_{N+1,k}(0)|_{\alpha,\sigma_0} + \frac{D_3}{k-(N+1)-\varep_*}\Big) e^{-\varep_* t}.
\eeqs
Thus
\begin{multline}\label{form2}
|R_kw_{N}(t)
- q_{N+1,k}(t-T_N)|_{\alpha,\sigma_0} \\
\le e^{\varep_* T_N} \Big(|R_k v_N(T_N)|_{\alpha,\sigma_0} + |q_{N+1,k}(0)|_{\alpha,\sigma_0} + \frac{D_3}{k-(N+1)-\varep_*}\Big) e^{-\varep_* t},
\quad\forall t\ge T_N. 
\end{multline}

\medskip
We define $$q_{N+1}(t)=\sum_{k=1}^\infty q_{N+1,k}(t-T_N), \quad t\in \R.$$

It follows from \eqref{forcexpand} that $f_{N+1}\in \mathcal{P}^{\infty,\sigma_0}$, and from the recursive assumptions that
$q_m, q_j\in \mathcal{P}^{\infty,\sigma_0}$ for $1\le m,j\le N$; then by \eqref{BEE}, we have $f_{N+1}-\sum_{m+j=N+1} B(q_m,q_j)\in \mathcal{P}^{\infty,\sigma_0}$. Following the same proof that shows  $q_1\in\mathcal{P}^{\infty,\sigma_0}$, we argue similarly to show that $q_{N+1}\in\mathcal{P}^{\infty,\sigma_0}$.

\medskip
\noindent{\bf Estimate of $v_{N+1}(t)$.}
From \eqref{form0} and \eqref{form1}, we immediately have
\beq\label{PNw}
|P_{N+1}(w_N(t)-q_{N+1}(t))|_{\alpha,\sigma_0} = \mathcal O(e^{-\varep_* t}).
\eeq

Squaring \eqref{form2} and summing over $k\ge N+2$, we obtain for $t\ge T_N$ that
\begin{align*}
&|(I-P_{N+1})(w_N(t)-q_{N+1}(t))|_{\alpha,\sigma_0}^2 = \sum_{k=N+2}^\infty |R_k (w_N(t)-q_{N+1}(t))|_{\alpha,\sigma_0}^2 \\
&\le 3 e^{2\varep_* T_N} \Big(\sum_{k=N+2}^\infty|R_k v_N(T_N)|_{\alpha,\sigma_0}^2 + \sum_{k=N+2}^\infty|R_kq_{N+1}(T_N)|_{\alpha,\sigma_0}^2 \\
&\quad + \sum_{k=N+2}^\infty\frac{D_3^2}{(k-(N+1)-\varep_*)^2}\Big)e^{-2\varep_* t}. 
\end{align*}

Since the last three sums are convergent, we deduce 
\beq\label{QNw}
|(I-P_{N+1})(w_N(t)-q_{N+1}(t))|_{\alpha,\sigma_0} = \mathcal O(e^{-\varep_* t}).
\eeq

From \eqref{PNw} and \eqref{QNw}, we have 
\beq\label{wNqN1}
|w_N(t)-q_{N+1}(t)|_{\alpha,\sigma_0}= \mathcal O(e^{-\varep_*t}).
\eeq
Therefore,
\beq \label{vN1ep}
|v_{N+1}(t)|_{\alpha,\sigma_0}
=|v_N(t)-e^{-(N+1)t} q_{N+1}(t)|_{\alpha,\sigma_0}= \mathcal O(e^{-(N+1+\varep_*)t}).
\eeq

Thanks to \eqref{wNqN1}, the polynomial $q_{N+1}(t)$ is independent of $\alpha$ and $\varep_*$. Therefore, \eqref{vN1ep} holds for any $\alpha\ge1/2$ and $\varep_*\in(0,\delta_{N+1}^*)$, which proves  ($\mathcal H2$) with $N+1$ replacing $N$.


\medskip
\noindent{\bf Establishing the ODE \eqref{unODE} for $n=N+1$.}
By our construction of the polynomials $q_{N+1,k}(t)$ above, and   
by \eqref{pode2} in Lemma \ref{polylem}, the polynomial $q_{N+1}(t)$ satisfies
\begin{align*} 
&\ddt R_k q_{N+1}(T_N+t)+(k-(N+1))R_kq_{N+1}(T_N+t) \\
&\quad =R_kf_{N+1}(T_N+t)- \sum_{m+j=N+1} R_k B(q_m(T_N+t),q_j(T_N+t)) 
\quad \forall k\ge 1,\quad \forall t\in\R.
\end{align*}
This yields,  for each $k\ge 1$,
\beq\label{Rkueq}
\ddt R_k u_{N+1}(t)+AR_ku_{N+1}(t) + \sum_{m+j=N+1} R_k B(u_m(t),u_j(t)) =R_kF_{N+1}(t),\quad \forall t\in\R.
\eeq

For any $\mu\ge 0$, since  $Au_{N+1}(t)$, $\sum_{m+j=N+1} B(u_m(t),u_j(t))$, $F_{N+1}(t)$ each is a $G_{\mu,\sigma_0}$-valued polynomial, then summing up equation \eqref{Rkueq} in $k$ gives 
\beqs
\ddt u_{N+1}+Au_{N+1}+ \sum_{m+j=N+1} B(u_m,u_j)=F_{N+1} \quad \text{in } G_{\mu,\sigma_0}.
\eeqs
Thus, the ODE \eqref{unODE} holds in $E^{\infty,\sigma_0}$  for $n=N+1$.

We have  established the existence of the desired polynomial, $q_{N+1}(t)$, which completes the recursive step, and hence, the proof of Theorem \ref{mainthm}.
\qed

\begin{remark}
By using the extra information \eqref{fep} in Remark \ref{betterem}, we can prove directly the remainder estimate \eqref{remdelta} for any $\varep\in(0,1)$, which, in fact, is expected by \eqref{uqa}.
Nonetheless, the above proof with specific $\varep\in(0,\delta^*_{N,\alpha})$  is more flexible and will  be easily adapted in section \ref{sec43} below to serve the proof of Theorem \ref{finitetheo}.
\end{remark}

\subsection{Proof of Corollary \ref{Vcor}}\label{sec42}
We follow the proof of Theorem \ref{mainthm}. Since $f_1\in \mathcal V$, there is $N_1\ge 1$ such that $f_1\in P_{N_1}H$. As a consequence, we see from \eqref{q1k} that $q_{1,k}=0$ for $k>N_1$.
Hence, $q_1(t)=\sum_{k=1}^{N_1} q_{1,k}(t-T_0)$ is a polynomial in $P_{N_1}H$.

For the recursive step, the functions $f_{N+1}$, $q_m$ and $q_j$ ($1\le m,j\le N$), in this case, are $\mathcal{V}$-valued polynomials.
Hence, by \eqref{BVV}, so is $f_{N+1}-\sum_{m+j=N+1} B(q_m,q_j)$. It follows that there are at most finitely many $k$ such that $q_{N+1,k}$ is nonzero.  Since each $q_{N+1,k}$ is an $\mathcal{V}$-valued polynomial, clearly $q_{N+1}$, as a finite sum of those, is a $\mathcal{V}$-valued polynomial.\qed

\subsection{Proof of Theorem \ref{finitetheo}}\label{sec43}
We follow the proof of Theorem \ref{mainthm} closely and make necessary modifications.
We prove part (i), while the proof of part (ii) is entirely similar to that in Theorem \ref{mainthm} and thus omitted.

The same notation $u_n(t)$, $F_n(t)$, $v_n(t)$, $\bar u_n(t)$ as in Theorem \ref{mainthm} is used here.

By \eqref{ffinite} with $N=1$,
\beq\label{ff1}
e^t |f(t)-F_1(t)|_{\alpha_*,\sigma_0}=\mathcal O(e^{-\delta_{1,\alpha_*} t}).
\eeq

Also, 
\beq\label{ffirst}
|f(t)|_{\alpha_*,\sigma_0}\le |f_1(t)|_{\alpha_*,\sigma_0}e^{-t}+|f(t)-f_1(t)e^{-t}|_{\alpha_*,\sigma_0}=\mathcal O(e^{-\lambda t}),\quad \forall \lambda\in(0,1).
\eeq

Using \eqref{ffirst} and by applying Proposition \ref{theo23}, we have
\beq \label{newua}
|u(t)|_{\alpha_*+1/2,\sigma_0}= \mathcal O(e^{-(1-\delta) t}),\quad \forall \delta\in(0,1),
\eeq 
\beq\label{newBu}
|B(u(t),u(t))|_{\alpha_*,\sigma_0}=\mathcal O(e^{-2(1-\delta) t}),\quad \forall \delta\in(0,1).
\eeq

\medskip
\noindent{\bf Base case: $N=1$.} 
Let $\alpha=\alpha_*$ and $\mu=\mu_*$.

In estimating $H_0(t)$ defined by \eqref{H0def}, the estimate \eqref{ch1}, resp., \eqref{ch2}, comes from \eqref{ff1}, resp. \eqref{newBu}.
Hence we obtain the bound \eqref{ch3} for $|H_0(T_0+t)|_{\alpha,\sigma_0}$.
Then the existence and definition of $q_1(t)$ are the same as in \eqref{q11def}, \eqref{q1k} and \eqref{q1def}.

Since $f_1\in \mathcal P^{\mu,\sigma_0}$, then the same proof gives 
 $q_1\in\mathcal P^{\mu+1,\sigma_0}$, see \eqref{R1q1poly}, \eqref{IRq1} and \eqref{Abn}. 
 
The remainder estimate \eqref{rem1} still holds true, which, for the current value $\alpha=\alpha_*$,  proves \eqref{ufinite} for $N=1$.
Also, the ODE \eqref{u1ODE} holds in $G_{\mu,\sigma_0}$ (for the current $\mu=\mu_*$).

If $N_*=1$, then the proof is finished here. We consider $N_*\ge 2$ now. 

\medskip
\noindent{\bf Recursive step.} Let $1\le N\le N_*-1$.
Assume there already exist $q_n\in \mathcal P^{\mu_n,\sigma_0}$ for $1\le n\le N$, such that
 \beq\label{farate0}
 | v_N(t)|_{\alpha_N,\sigma_0}=\mathcal O(e^{-(N+\varep)t}),\quad\forall \varep\in(0,\delta_N^*),
 \eeq
and \eqref{unODE} holds in $ G_{\mu_N,\sigma_0}$ for $n=1,2,3,\ldots,N$.

Let
\beq\label{almu}
\alpha=\alpha_{N+1}=\alpha_N-1/2\ge 1/2, \quad \mu=\mu_{N+1}=\mu_N-1/2\ge \alpha\ge 1/2.
\eeq

Note for $n=1,2,\ldots,N$ that $\mu_n\ge \alpha_n\ge 1/2$ and both $\mu_n$, $\alpha_n$ are decreasing, 
hence,  
\beq \label{uqn}
u_n(t),q_n(t)\in G_{\mu_n,\sigma_0} \subset G_{\mu_N,\sigma_0}=G_{\mu+1/2,\sigma_0} \subset G_{\alpha_N,\sigma_0}=G_{\alpha+1/2,\sigma_0},\quad\forall t\in\R.\eeq

Rewrite \eqref{farate0}  as 
\beq\label{favNrate}
| v_N(t)|_{\alpha+1/2,\sigma_0}=\mathcal O(e^{-(N+\varep)t}),\quad \forall \varep\in(0,\delta_N^*).
\eeq

We now construct a polynomial $q_{N+1}\in\mathcal P^{\mu+1,\sigma_0}$ such that 
\eqref{ufinite} holds true with $N+1$ replacing $N$, and the ODE \eqref{unODE}, with $n=N+1$, holds in $G_{\mu_{N+1},\sigma_0}=G_{\mu,\sigma_0}$. 

We proceed with the construction of $q_{N+1}(t)$ as in the proof of Theorem \ref{mainthm} using the specific values of $\alpha$ and $\mu$ in \eqref{almu}.

Note that equation \eqref{vNeq} for $v_N$ still holds in the weak sense as in Definition \ref{lhdef}.

\medskip
$\bullet$ We check the estimate \eqref{hNo} for the function $h_N(t)$ defined by \eqref{hdef}.

Let $\varep_*\in(0,\delta^*_{N+1})$.
By \eqref{ffinite} with $N+1$ replacing $N$, we have
\beq \label{faFtil}
|\tilde F_{N+1}(t)|_{\alpha,\sigma_0}=\mathcal O(e^{-(N+1+\delta_{N+1})t})
=\mathcal O(e^{-(N+1+\varep_*)t}).
\eeq

Thanks to \eqref{uqn} and Lemma \ref{nonLem}, estimate \eqref{Bmj} stays the same as
\beq \label{faBmj}
\sum_{\stackrel{1\le m,j\le N}{m+j\ge  N+2}} |B(u_m,u_j)|_{\alpha,\sigma_0}
=\mathcal O(e^{-(N+1+\varep_*)t}).
\eeq

Again, take $\varep\in(\varep_*,\delta^*_{N+1})\subset (0,\delta^*_N)$ in \eqref{favNrate} and $\delta=\varep-\varep_*\in(0,1)$.
Then we have
\beq \label{faubu}
|\bar u_N(t)|_{\alpha+1/2,\sigma_0}=|\bar u_N(t)|_{\alpha_N,\sigma_0}=\mathcal O(e^{-(1-\delta)t}),
\eeq
and by \eqref{newua}
\beqs
|u(t)|_{\alpha+1/2,\sigma_0}\le |u(t)|_{\alpha_*,\sigma_0}=\mathcal O(e^{-(1-\delta)t}).
\eeqs

By Lemma \ref{nonLem} and estimates \eqref{favNrate}, \eqref{faubu}, it follows that
$$|B(v_N,u)|_{\alpha,\sigma_0},|B(\bar u_N,v_N)|_{\alpha,\sigma_0}=\mathcal O(e^{-(N+\varep+1-\delta)t})=\mathcal O(e^{-(N+1+\varep_*)t}).$$
Therefore, by \eqref{faFtil}, \eqref{faBmj} and \eqref{faubu}, we again obtain \eqref{hNo}.

\medskip
$\bullet$  The same construction of  $q_{N+1}(t)$ now goes through. Indeed, since $f_{N+1}\in \mathcal P^{\mu,\sigma_0}$, by \eqref{uqn} and the fact that  $q_m, q_j\in \mathcal P^{\mu+1/2,\sigma_0}$ for $1\le m,j\le N$, it then follows from \eqref{BGG} that
$$f_{N+1}-\sum_{m+j=N+1} B(q_m,q_j)\in \mathcal P^{\mu,\sigma_0}.$$

The same proof as for the case of $q_1$ then yields that  $q_{N+1}\in\mathcal P^{\mu+1,\sigma_0}$.

\medskip
$\bullet$ 
For the estimate of $v_{N+1}(t)$, same arguments yield
\beqs
|v_{N+1}(t)|_{\alpha,\sigma_0}
=|v_N(t)-e^{-(N+1)t} q_{N+1}(t)|_{\alpha,\sigma_0}= \mathcal O(e^{-(N+1+\varep_*)t}). 
\eeqs

Since this holds for any $\varep_*\in(0,\delta_{N+1}^*)$, the remainder estimate \eqref{ufinite} holds true with $N+1$ replacing $N$.


\medskip
$\bullet$ As for the ODE \eqref{unODE} with $n=N+1$, the proof is unchanged from that of Theorem \ref{mainthm}, noting that the ODE now holds in the corresponding space $G_{\mu,\sigma_0}$.

\medskip
We have proved that the polynomial $q_{N+1}$ has the desired properties. This completes the recursive step, and hence, the proof of Theorem \ref{finitetheo}.\qed

\section*{}
\noindent\textbf{\large Acknowledgements.} 
The authors would like to thank Peter Constantin for prompting the question of extending the Foias-Saut theory to the case of non-potential forces. The authors are also grateful to Ciprian Foias for his encouragement towards this work and helpful discussions,
as well as to Animikh Biswas for stimulating discussions. L.H. gratefully acknowledges the support for his research by the NSF grant DMS-1412796

\bibliography{paperbaseall}{}
\bibliographystyle{abbrv}

\end{document}